\documentclass[11pt]{article}
\usepackage{amsmath, amssymb, amsthm}
\usepackage{graphicx}
\usepackage{epsfig}
\usepackage{epstopdf}
\usepackage{mathrsfs}

\renewcommand{\baselinestretch}{1.25}

 \newtheorem{thm}{Theorem}[section]
 \newtheorem{cor}[thm]{Corollary}
 \newtheorem{lem}[thm]{Lemma}
 \newtheorem{prop}[thm]{Proposition}
 \theoremstyle{definition}
 \newtheorem{defn}[thm]{Definition}
 \newtheorem{exmp}[thm]{Example}

 \newtheorem{rem}[thm]{Remark}

\numberwithin{equation}{section}

\newcommand{\Nat}{\mathbb{N}}
\newcommand{\set}[1]{\left\{#1\right\}}
\newcommand{\Set}[2]{\set{#1\ \vert\ #2}}

\newcommand{\Map}[3]{#1\, :\, #2\longrightarrow #3}
 
\date{}

\begin{document}

\title{On certain classes of graceful lobsters}
\author{Shamik Ghosh\thanks{Department of Mathematics, Jadavpur University, Kolkata-700032, India. sghosh@math.jdvu.ac.in}}

\maketitle

\renewcommand{\baselinestretch}{1}
\begin{abstract}
A (simple undirected) graph $G=(V,E)$ with $m$ edges is graceful if it has a distinct vertex labeling $\Map{f}{V}{\set{0,1,2,3,\ldots,m}}$\ \ which induces a set of distinct edge labels\ \ $\Set{|f(u)-f(v)|}{uv\in E,\ u,v\in V}$. The famous Ringel-Kotzig conjecture \cite{Kot,Rosa} is that all trees are graceful. The base of a tree $T$ is  obtained from $T$ by deleting its one-degree vertices. A caterpillar is a tree whose base is a path and a lobster is a tree whose base is a caterpillar. Paths and caterpillars are known to be graceful. Next it was conjectured by Bermond \cite{Ber} that all lobsters are graceful. In this paper we describe various methods of joining graceful graphs and $\alpha$-labeled graphs using the adjacency matrix characterization that initiated by Bloom \cite{Bloom} and others. We apply these results to obtain some classes of graceful lobsters and indicate how to obtain some others.

\vspace{0.5em}\noindent

\end{abstract}

\renewcommand{\baselinestretch}{1.5}



\section{Introduction}
	
A graph labeling is an assignment of integers to the vertices or edges, or both, subject to certain conditions. Let $G=(V,E)$ be a simple undirected graph with $n$ vertices and $m$ edges. In 1967, Rosa \cite{Rosa} called a map $\Map{f}{V(G)}{\set{0,1,2,\ldots ,m}}$ a $\beta$-\emph{labeling} (or, $\beta$-\emph{valuation}) of $G$ if $f$ is injective and when each edge $uv$ is assigned the label $|f(u)-f(v)|$, the resulting edge labels are distinct. Later on Golomb \cite{Gol} called such labeling {\em graceful}. If in addition to that there exists $k\in\set{0,1,2,\ldots ,m}$ such that for every edge $uv\in E$, either $f(u)\leqslant k<f(v)$ or $f(v)\leqslant k<f(u)$, then $G$ is said to have an $\alpha$-\emph{labeling} (or, $\alpha$-\emph{valuation}) $f$. Let us call the number $k$, a {\em critical number} of the $\alpha$-labeling $f$. An $\alpha$-labeling $f$ is called {\em complete} \cite{Kot} if $f$ is bijective. In the sequel, by a graph we mean a simple undirected graph. A graph $G$ is a {\em graceful graph} if it has a graceful labeling. The famous Ringel-Kotzig conjecture \cite{Kot,Ring,Rosa} is that {\em all trees are graceful}. A lot of research papers were published on graceful graphs and their other variations \cite{BD,HKR,Kot}. Many special classes of trees are known to be graceful. A comprehensive survey and references of graceful graphs and their variations can be found in \cite{Gal}. Another survey and the adjacency matrix characterization of graceful graphs are available in \cite{Bloom}.

\vspace{1em} 
The {\em base} of a tree $T$ is a tree which is obtained from $T$ by deleting its pendant (one-degree) vertices. A {\em caterpillar} is a tree whose base is a path and a {\em lobster} is a tree whose base is a caterpillar. Paths and caterpillars are known to be graceful \cite{Rosa}. In 1979 Bermond \cite{Ber} conjectured that all lobsters are graceful. Several special classes of lobsters are shown to be graceful \cite{MP,DM,Ng,SJ,WJL}.

\vspace{1em}
In this paper, we first describe adjacency matrices of graceful graphs and graphs with $\alpha$-labeling. Using these matrix representations we obtain several methods of joining graceful graphs and graphs with $\alpha$-labeling. Next we define three special classes of lobsters, namely, pairwise similar, pairwise linked and pairwise balanced (cf. Definitions \ref{lobd1}, \ref{defgr3}, \ref{lobpb2}). We show that pairwise linked lobsters are graceful, pairwise similar lobsters are graceful under certain conditions and pairwise balanced lobsters are $\alpha$-labeled and hence graceful. Finally we present an illustration which indicates how to obtain many other graceful lobsters which do not belong to the above classes.

\section{Adjacency matrices of graceful graphs}

Let $G=(V,E)$ be a graph with $n$ vertices and $m$ edges. Then by the definition of graceful graph, it is clear that $G$ can never be graceful if $n>m+1$ as in this case there is no injective map $\Map{f}{V}{\set{0,1,2,\ldots ,m}}$. Thus if $G$ is graceful, then $m\geqslant n-1$. Note that this condition is always satisfied if $G$ is connected. Now let $n\leqslant m+1$. We include $m+1-n$ isolated vertices in $G$ and the graph thus obtained is denoted by $\widehat{G}$. Clearly, the graph $G$ is graceful if and only if $\widehat{G}$ is also so. A graceful graph $G$ is a {\em completely graceful graph} if there is a bijective  $\beta$-labeling of $G$, i.e., if $n=m+1$ (or equivalently, if $G=\widehat{G}$).\footnote{In \cite{Bloom}, Bloom called the graph $G$ {\em fully augmented} if $G=\widehat{G}$.}  Note that any graceful tree is a completely graceful graph and every connected completely graceful graph is a tree.

\begin{defn}\label{d1}
Let $A$ be an $m\times n$ matrix. The {\em box-value} of the position $(i,j)$ of $A$ is $m+j-i$. A {\em diagonal} of $A$ is a set of positions of $A$ with the same box-value, i.e., for each $c=1,2,\ldots ,m+n-1$, the set $d_c=\Set{(i,j)}{m+j-i=c}$ is a  diagonal. An $m\times n$ binary matrix (i.e., a $0$-$1$ matrix) $A$ is said to be {\em graceful} if every diagonal of $A$ contains at most one $1$ and $A$ is said to be {\em completely graceful} if every diagonal of $A$ contains exactly one $1$ (except the principal diagonal of an adjacency matrix of a graph where all the entires are $0$). 
\end{defn}

\begin{thm}\label{t0}\cite{Bloom,HW}
Let $G$ be a graph with $n$ vertices and $m$ edges such that $n=m+1$. Then $G$ is completely graceful if and only if the vertices of $G$ can be arranged in such a way that its adjacency matrix becomes completely graceful.
\end{thm}

The following theorem is an immediate generalization of the above result. The proof is omitted as it is similar to that of Theorem \ref{t0}.

\begin{thm}\label{t1}
Let $G$ be a graph with $n$ vertices and $m$ edges. Then $G$ is graceful if and only if $n\leqslant m+1$ and the vertices of $\widehat{G}$ can be arranged in such a way that its adjacency matrix becomes graceful.
\end{thm}

\begin{cor}\label{c1}
A tree $T$ is graceful if and only if the vertices of $T$ can be arranged in such a way that its adjacency matrix is completely graceful.
\end{cor}

\begin{rem}\label{rem1}
It is important to note that for a graceful graph $G$, a graceful adjacency matrix, say, $A$ of $\widehat{G}$ is obtained by arranging the vertices of $\widehat{G}$ according to the increasing order of their labels in a graceful labeling of $\widehat{G}$. With this arrangement $A$ is known as a {\em canonical adjacency matrix} \cite{Bloom} of $\widehat{G}$ (or, of $G$ when $G$ is completely graceful). For example, the matrix $A$ in Example \ref{e1} is a canonical adjacency matrix of the graceful tree $G$ in Figure \ref{fig1}(left).
\end{rem}

Now let $G=(V,E)$ be a graph with an $\alpha$-labeling $f$ and a critical number $k$. It is known that $G$ is bipartite \cite{Kot}. In fact, if $V_1=\Set{v\in V}{f(v)\leqslant k}$ and $V_2=\Set{v\in V}{f(v)>k}$, then since $f$ is an $\alpha$-labeling, $V_1$ and $V_2$ forms a bipartition of $V$ in $G$. For a bipartite graph $B=(X,Y,E)$ with partite sets $X$ and $Y$, the submatrix $A$ of the adjacency matrix of $B$ containing rows corresponding to the vertices of $X$ and columns corresponding to the vertices of $Y$ is known as the {\em biadjacency matrix} of $B$ (see Figure \ref{f1}). Note that a graph $G$ is bipartite if and only if $\widehat{G}$ is also bipartite.

\vspace{-0.5em}
\begin{figure}[h]
$$\begin{array}{c|c|c|c}
\multicolumn{1}{c}{} & \multicolumn{1}{c}{X} & \multicolumn{2}{c}{Y\ }\\ 
\cline{2-3}

X & \mathbf{0} & A &\\ 
 \cline{2-3}

Y & A^T & \mathbf{0} &\\
\cline{2-3}
\end{array}$$
\caption{The biadjacency matrix $A$ in the adjacency matrix of a bipartite graph $B=(X,Y,E)$}\label{f1}
\end{figure}

\begin{thm}\label{t2}
Let $G$ be a graph with $n$ vertices and $m$ edges. Then $G$ has an $\alpha$-labeling [a complete $\alpha$-labeling] with a critical number $k$ if and only if $n\leqslant m+1$ [respectively, $n=m+1$], $G$ is bipartite and there is a bipartition of $\widehat{G}$ where the vertices of $\widehat{G}$ can be arranged in such a way that its biadjacency matrix $A$ becomes graceful [respectively, completely graceful] and $A$ has $k+1$ rows.
\end{thm}

\begin{proof}
Let $G=(V,E)$ be a bipartite graph with $n$ vertices and $m$ edges. Suppose $n\leqslant m+1$ and there is a bipartition $\widehat{V}=\widehat{V}_1\cup\widehat{V}_2$ of $\widehat{G}=(\widehat{V},E)$ where the vertices of $\widehat{G}$ can be arranged in such a way that its biadjacency matrix $A$ is graceful and $A$ has $k+1$ rows. Then the adjacency matrix of $\widehat{G}$ is also graceful and hence by Theorem \ref{t1}, $\widehat{G}$ is graceful with the $\beta$-labeling $\widehat{f}$ such that $\widehat{f}(v_i)=i$ for all $i=0,1,2,\ldots,m$, where $\widehat{V}=\set{v_0,v_1,v_2,\ldots,v_m}$. Let $\widehat{V}_1=\set{v_0,v_1,v_2,\ldots,v_k}$ and $\widehat{V}_2=\set{v_{k+1},v_{k+2},\ldots,v_m}$. Then for every edge $uv\in E$, ($u\in V_1$, $v\in V_2$), we have $\widehat{f}(u)\leqslant k<\widehat{f}(v)$. Thus $\widehat{G}$ and hence $G$ has an $\alpha$-labeling, say, $f$ with a critical number $k$. 

\vspace{0.5em}
Conversely, let $G=(V,E)$ be a graph with $n$ vertices and $m$ edges and it has an $\alpha$-labeling $f$ with a critical number $k$. Then $G$ is bipartite and $n\leqslant m+1$. Let $\widehat{G}=(\widehat{V},E)$. Define a labeling  $\Map{\widehat{f}}{\widehat{V}}{\set{0,1,2,\ldots ,m}}$ to the vertices of $\widehat{G}$ such that $\widehat{f}(u)=f(u)$ for all $u\in V$. Label the other vertices of $\widehat{G}$ arbitrarily so that $\widehat{f}$ becomes bijective. Then $\widehat{f}$ is an $\alpha$-labeling of $\widehat{G}$ with the same critical number $k$ since $G$ and $\widehat{G}$ have the same set of edges. Let us denote the vertex $v\in \widehat{V}$ by $v_i$ if $\widehat{f}(v_i)=i$. Let $\widehat{V}_1=\set{v_0,v_1,v_2,\ldots,v_k}$ and $\widehat{V}_2=\set{v_{k+1},v_{k+2},\ldots,v_m}$. Then $\widehat{V}=\widehat{V}_1\cup\widehat{V}_2$ is a bipartition of $\widehat{V}$ and the biadjacency matrix, say, $A$ of $\widehat{G}$ corresponding to this partition (arranging the vertices according to the increasing order of their indices) is graceful and $A$ has $k+1$ rows as $\widehat{f}$ is an $\alpha$-labeling of $\widehat{G}$ with a critical number $k$.

Finally, it is clear that any $\alpha$-labeling $f$ of $G$ is a complete $\alpha$-labeling if and  only if $n=m+1$.
\end{proof}

\begin{cor}\label{c2}
Let $T$ be a tree. Then $T$ has a complete $\alpha$-labeling if and only if there is a bipartition of $T$ where the vertices of $T$ can be arranged in such a way that its biadjacency matrix $A$ becomes completely graceful.
\end{cor}

\begin{figure}[h]
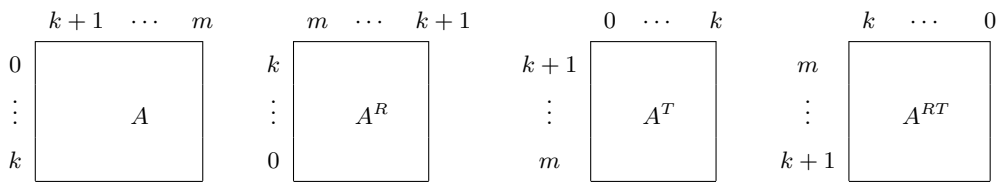

\begin{center}
{\footnotesize $\begin{array}{c|llr|r}
\multicolumn{1}{c}{} & \multicolumn{1}{c}{k+1} & \cdots & \multicolumn{2}{c}{m}\\
\cline{2-4} 
0 & & & &\\ 
\vdots & & A & &\\ 
k & & &\\
\cline{2-4}
\end{array}$\quad $\begin{array}{c|llr|r}
\multicolumn{1}{c}{} & \multicolumn{1}{c}{m} & \cdots & \multicolumn{2}{c}{k+1}\\
\cline{2-4} 
k & & & &\\ 
\vdots & & A^R & &\\ 
0 & & &\\
\cline{2-4}
\end{array}$\quad $\begin{array}{c|llr|l}
\multicolumn{1}{c}{} & \multicolumn{1}{c}{0} & \cdots & \multicolumn{2}{c}{k}\\
\cline{2-4} 
k+1 & & & &\\ 
\vdots & & A^T & &\\ 
m & & &\\
\cline{2-4}
\end{array}$\quad $\begin{array}{c|lcr|r}
\multicolumn{1}{c}{} & \multicolumn{1}{c}{k} & \cdots & \multicolumn{2}{c}{0}\\
\cline{2-4} 
m & & & &\\ 
\vdots & & A^{RT} & &\\ 
k+1 & & &\\
\cline{2-4}
\end{array}$}
\caption{The matrices (i) $A$, (ii) $A^R$, (iii) $A^T$ and (iv) $A^{RT}$}\label{fig2}
\end{center}
\end{figure}

\begin{defn}\label{defcan}
Let $G=(V,E)$ be a graph with $n$ vertices and $m=n-1$ edges. Suppose $G$ has a complete $\alpha$-labeling $f$ with a critical number $k$. Define $\Map{f^*}{V}{\set{0,1,2,\ldots ,m}}$ by $f^*(v)=k-f(v) \pmod{n}$. Then $f^*(v)$ is again a complete $\alpha$-labeling of $G$ and it is known as the {\em inverse} $\alpha$-labeling of $f$ \cite{HKR,RosaT,Sh}. Let $A$ be the completely graceful biadjacency matrix of $G$ consisting of $k+1$ rows corresponding to the vertices $v\in V$ such that $0\leqslant f(v)\leqslant k$ and $m-k$ columns corresponding to the vertices $v\in V$ such that $k+1\leqslant f(v)\leqslant m$, where vertices are arranged according to the increasing order of their labels (cf. Figure \ref{fig2}(i)). The matrix $A$ is called a {\em canonical biadjacency matrix} of $G$. Now if we arrange the vertices $v$ of $G$ in the matrix $A$ according to the increasing order of $f^*(v)$ keeping the same number of rows and columns and obtain a matrix $A^R$ (cf. Figure \ref{fig2}(ii)), then $A^R$ is again a completely graceful biadjacency matrix of $G$. Also it is easy to see that the transpose of a completely graceful matrix is again a completely graceful matrix. Thus $A^T$ (cf. Figure \ref{fig2}(iii)) and $A^{RT}=(A^R)^T$ (cf. Figure \ref{fig2}(iv)) are also completely graceful biadjacency matrix of $G$.
\end{defn}

\section{Joining graceful graphs}

In this section we describe some methods of joining graceful graphs or graphs with $\alpha$-labeling which we will use to obtain some classes of graceful lobsters. Some of these results in some other form or in some special form appeared in the literature of graceful graphs \cite{Bloom,HKR,SZ}, but for the sake of completeness and further use we include the sketches of proofs.

\begin{defn}\label{double}
Let $G=(V,E)$ be a graceful graph with $n$ vertices, $m$ edges and a $\beta$-labeling $f$. Suppose $V\subseteq\set{v_0,v_1,\ldots,v_m}$, where $f(v_i)=i$ for all $v_i\in V$. Let $G_1=(V_1,E_1)$ be another (isomorphic) copy of $G$ with a $\beta$-labeling $f_1$ such that $V_1=\Set{u_i}{v_i\in V}$ and $f_1(u_i)=i=f(v_i)$ for all $v_i\in V$.  Now let us fix some $j\in\Set{i}{v_i\in V}$. Let $G_2$ be the graph obtained from $G$ and $G_1$ by joining the vertices $v_j$ and $u_j$ with an edge. Then $G_2$ is called a {\em double} of $G$ at $j$.
\end{defn}

The above definition generalizes the concept of doubled graceful trees in \cite{Bloom}.

\begin{prop}\label{p1}
Any double of a graceful graph with $m$ edges has an $\alpha$-labeling with a critical number $m$.
\end{prop}

\begin{proof}
Let $G=(V,E)$ be a graceful graph with $m$ edges and a $\beta$-labeling $f$. Then it follows from Theorem \ref{t1} and Remark \ref{rem1}, a canonical adjacency matrix, say, $A$ of $\widehat{G}$ is graceful, where the vertices are arranged according to the increasing order of their labels. Suppose $G_2$ is obtained from $G$ as in Definition \ref{double}. Let $A_1$ be the matrix obtained from $A$ by adding a $1$ in the position $(j,j)$. Let 
$$A_2\ =\ \begin{array}{|c|c|}
\hline
\mathbf{0} & A_1\\ 
\hline
A_1^T & \mathbf{0}\\
\hline
\end{array}$$
Then it is easy to see that $A_2$ is the adjacency matrix of $\widehat{G}_2$ and $A_1$ is the biadjacency matrix of $\widehat{G}_2$. Moreover $A_1$ is graceful as it includes only one $1$ at the principal diagonal of $A$ which contains only $0$'s in $A$. Thus by Theorem \ref{t2}, $\widehat{G}_2$ and hence $G_2$ has an $\alpha$-labeling with a critical number $m$.
\end{proof}

In \cite{Bloom}, Bloom observed that a double of a graceful tree is graceful. For our further use it is important to note that it has a complete $\alpha$-labeling.

\begin{cor}\label{dgrfl}
Any double of a completely graceful graph with $m$ edges has a complete $\alpha$-labeling with a critical number $m$. Hence any double of a graceful tree with $m$ edges has a complete $\alpha$-labeling with a critical number $m$.
\end{cor}

\vspace{-0.5em}
\begin{figure}[h]
\begin{center}
\includegraphics[scale=0.65]{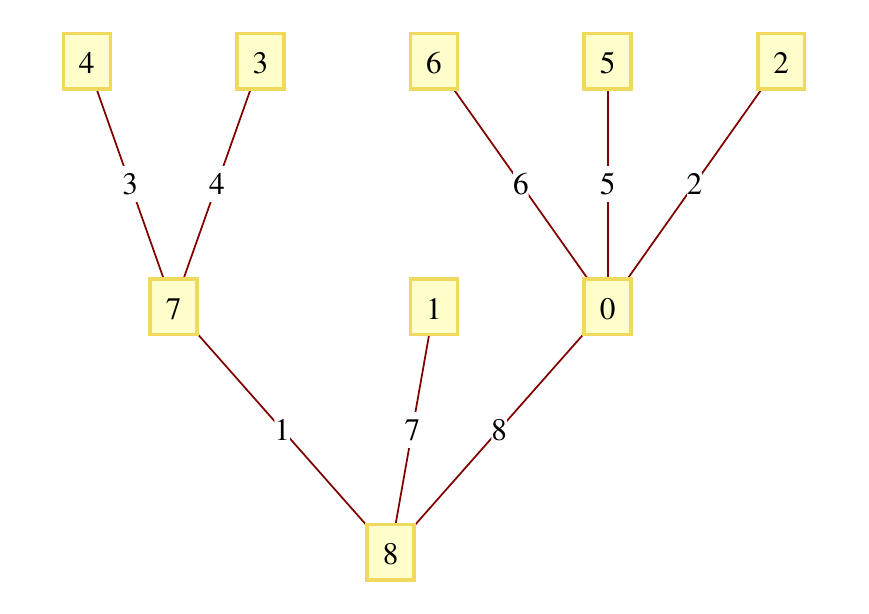}\qquad \includegraphics[scale=0.75]{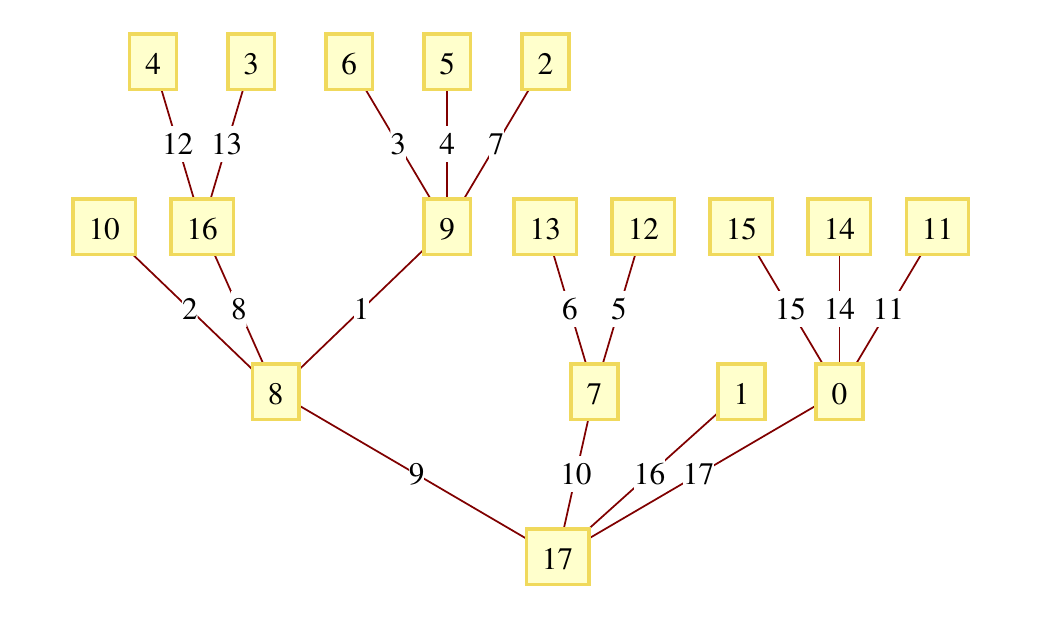}
\caption{The graceful graph $G$ (left) and its double $G_2$ (right) at $8$ in Example \ref{e1}}\label{fig1}
\end{center}
\end{figure}

\begin{exmp}\label{e1}
Let $G$ be the graceful tree in Figure \ref{fig1}(left). Let us fix the vertex labeled $8$ in $G$ and obtain a double $G_2$ of $G$ by joining this vertex with its corresponding vertex in a copy of $G$ by an edge (cf. Figure \ref{fig1}(right)). Then $G_2$ has a complete $\alpha$-labeling with a critical number $8$. A canonical adjacency matrix $A$ of $G$ and a canonical biadjacency matrix $A_1$ of $G_2$ are given by
{\footnotesize $$A=\begin{array}{l|lllllllll|}
\multicolumn{1}{c}{} & 0 & 1 & 2 & 3 & 4 & 5 & 6 & 7 & \multicolumn{1}{c}{8}\\
\cline{2-10}
0 & 0 & 0 & 1 & 0 & 0 & 1 & 1 & 0 & 1 \\
1 & 0 & 0 & 0 & 0 & 0 & 0 & 0 & 0 & 1 \\
2 & 1 & 0 & 0 & 0 & 0 & 0 & 0 & 0 & 0 \\
3 & 0 & 0 & 0 & 0 & 0 & 0 & 0 & 1 & 0 \\
4 & 0 & 0 & 0 & 0 & 0 & 0 & 0 & 1 & 0 \\
5 & 1 & 0 & 0 & 0 & 0 & 0 & 0 & 0 & 0 \\
6 & 1 & 0 & 0 & 0 & 0 & 0 & 0 & 0 & 0 \\
7 & 0 & 0 & 0 & 1 & 1 & 0 & 0 & 0 & 1 \\
8 & 1 & 1 & 0 & 0 & 0 & 0 & 0 & 1 & 0 \\
\cline{2-10}
\end{array}\quad A_1=\begin{array}{l|ccccccccc|}
\multicolumn{1}{c}{} & 9 & 10 & 11 & 12 & 13 & 14 & 15 & 16 &\multicolumn{1}{c}{17}\\
\cline{2-10}
0 & 0 & 0 & 1 & 0 & 0 & 1 & 1 & 0 & 1 \\
1 & 0 & 0 & 0 & 0 & 0 & 0 & 0 & 0 & 1 \\
2 & 1 & 0 & 0 & 0 & 0 & 0 & 0 & 0 & 0 \\
3 & 0 & 0 & 0 & 0 & 0 & 0 & 0 & 1 & 0 \\
4 & 0 & 0 & 0 & 0 & 0 & 0 & 0 & 1 & 0 \\
5 & 1 & 0 & 0 & 0 & 0 & 0 & 0 & 0 & 0 \\
6 & 1 & 0 & 0 & 0 & 0 & 0 & 0 & 0 & 0 \\
7 & 0 & 0 & 0 & 1 & 1 & 0 & 0 & 0 & 1 \\
8 & 1 & 1 & 0 & 0 & 0 & 0 & 0 & 1 & \mathbf{1} \\
\cline{2-10}
\end{array}
$$}
\end{exmp}

It is easy to join any (finite) collection of graphs with $\alpha$-labeling with no common vertices to form a graph with $\alpha$-labeling.

\begin{prop}\label{p2}
Let $\set{G_1,G_2,\ldots,G_r}$ be a collection of graphs such that for each $i=1,2,\ldots,r$, $G_i=(V_i,E_i)$ has an $\alpha$-labeling with a critical number $k_i$ and $V_i\cap V_j=\emptyset$ for all $i\neq j$. Then the graph $G=G_1+ G_2+\cdots + G_r$ (disjoint union of graphs $G_1,G_2,\ldots,G_r$) has an $\alpha$-labeling with a critical number $k=k_1+k_2+\cdots +k_r+r-1$.
\end{prop}

\begin{proof}
For each $i=1,2,\ldots,r$, let $A_i$ be a canonical biadjacency matrix of $\widehat{G}_i$. Then $A_i$ has $k_i+1$ rows (cf. Theorem \ref{t2}). Consider the following matrix $A$:

$$A=\left(
\begin{array}{llll}
\mathbf{0} & \ldots & \mathbf{0} & A_1 \\
\mathbf{0} & \ldots & A_2 & \mathbf{0} \\
\vdots & \reflectbox{$\ddots$} & \vdots & \vdots \\
A_r & \ldots & \mathbf{0} & \mathbf{0}
\end{array}
\right)$$

Then it is clear that $A$ is a graceful biadjacency matrix of $\widehat{G}$ and $A$ has $k_1+k_2+\cdots +k_r+r$ rows. Thus $\widehat{G}$ and hence $G$ has an $\alpha$-labeling with a critical number $k=k_1+k_2+\cdots +k_r+r-1$.
\end{proof}

Now we wish to join a collection of vertex-disjoint graphs with complete $\alpha$-labelings in the form of a chain by adding edges between certain vertices of consecutive graphs in the collection to obtain a bigger graph with a complete $\alpha$-labeling. For brievity we will just indicate the completely graceful biadjacency matrices or adjacency matrices of the resulting graphs. 

\begin{prop}\label{p3}
Let $\set{G_1,G_2,\ldots,G_r}$ be a collection of graphs such that for each $i=1,2,\ldots,r$, $G_i=(V_i,E_i)$ has a complete $\alpha$-labeling with a critical number $k_i$ and $V_i\cap V_j=\emptyset$ for all $i\neq j$. Let us denote a vertex in the graph $G_i$ with the label $\lambda$ by $\lambda_i$ and for convenience we denote $(k_i)_i$ by $k_i$. Let $G=(V,E)$ be a graph obtained by joining the vertices $k_i$ and $m_{i+1}$ with an edge for each $i=1,2,\ldots,r-1$ (cf. Figure \ref{figkm1}). Then $G$ has a complete $\alpha$-labeling with a critical number $k=k_1+k_2+\cdots +k_r+r-1$.
\end{prop}

\begin{proof}
For each $i=1,2,\ldots,r$, let $A_i$ be a canonical biadjacency matrix of $G_i$ (cf. Definition \ref{defcan}). Then a completely graceful biadjacency matrix $A$ of $G$ is given by
{\footnotesize $$A=\quad\begin{array}{c|llr|rr|llr|lll|}
\multicolumn{1}{c}{} & \multicolumn{1}{c}{k_r+1} & & \multicolumn{1}{l}{m_r} & \cdots & \multicolumn{1}{r}{} & \multicolumn{1}{c}{k_2+1} & & \multicolumn{1}{l}{m_2} & \multicolumn{1}{c}{k_1+1} & & \multicolumn{1}{l}{m_1} \\
\cline{2-12} 
0_1 & & & & & & & & & & & \\ 
 & & & & \cdots & & & & & & A_1 & \\
k_1 & & & & & & & & 1 & & & \\
\cline{2-12}
0_2 & & & & & & & & & & & \\ 
 & & & & \cdots & & & A_2 & & & & \\
k_2 & & & & & 1 & & & & & & \\
\cline{2-12}
\vdots & & \vdots & & & \reflectbox{$\ddots$} & & \vdots & & & \vdots & \\ 
k_{r-1} & & \vdots & 1 & \reflectbox{$\ddots$} & & & \vdots & & & \vdots & \\
\cline{2-12}
0_r & & & & & & & & & & & \\ 
 & & A_r & & \cdots & & & & & & & \\
k_r & & & & & & & & & & & \\
\cline{2-12}
\end{array}$$}
\end{proof}

\begin{figure}[ht]
\begin{center}
\includegraphics[scale=0.5]{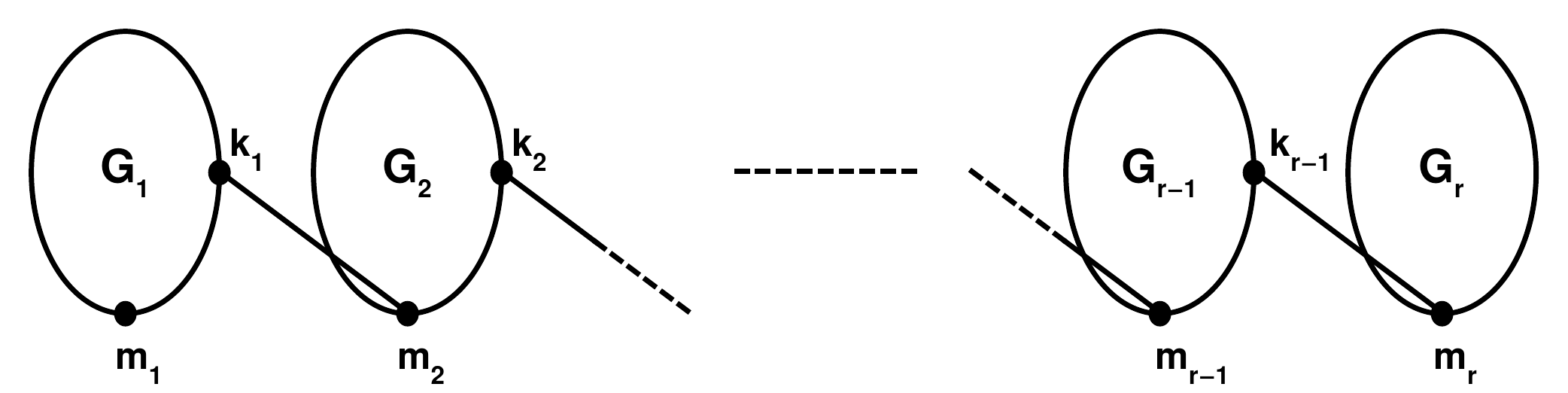}
\caption{The graph $G$ in Proposition \ref{p3}}\label{figkm1}
\end{center}
\end{figure}

Now if we replace the submatrix $A_i$ by $A^T_i$ (cf. Figure \ref{fig2}(iii)) for all odd $i\in\set{1,2,\ldots,r}$ in the matrix $A$ described in the proof of Proposition \ref{p3}, then we have the following result.

\begin{prop}\label{p4}
Let $\set{G_1,G_2,\ldots,G_r}$ be a collection of graphs defined as in Proposition \ref{p3}. Let $G=(V,E)$ be a graph obtained by joining the vertices $m_i$ and $m_{i+1}$ with an edge for each odd $i\in\set{1,2,\ldots,r-1}$ and joining the vertices $k_i$ and $k_{i+1}$ with an edge for each even $i\in\set{2,3,\ldots,r-1}$ (cf. Figure \ref{figkm2}(left)). Then $G$ has a complete $\alpha$-labeling with a critical number $k=k_1+k_2+\cdots +k_r+r-1$. Moreover, if $m_i-k_i=m_{i+1}-k_{i+1}$ for all even $i\in\set{2,3,\ldots,r-1}$, then the graph $H$ has a complete $\alpha$-labeling with a critical number $k=k_1+k_2+\cdots +k_r+r-1$, where the graph $H$ is obtained by joining the vertices $m_i$ and $m_{i+1}$ with an edge for each $i=1,2,\ldots,r-1$ (cf. Figure \ref{figkm2}(right)).
\end{prop}

\begin{proof}
A completely graceful biadjacency matrix $A$ of $G$ is obtained from the matrix $A$ described in the proof of Proposition \ref{p3} by replacing the submatrices $A_i$ by $A^T_i$ (cf. Figure \ref{fig2}(iii)) for all odd $i\in\set{1,2,\ldots,r}$ and a completely graceful biadjacency matrix $B$ of $H$ is given by
{\footnotesize $$B=\quad\begin{array}{c|l|llr|llr|llr|lll|}
\multicolumn{1}{c}{} & \multicolumn{1}{c}{\cdots} & \multicolumn{1}{c}{k_4+1} & & \multicolumn{1}{l}{m_4} & 0_3 & & \multicolumn{1}{r}{k_3} & \multicolumn{1}{c}{k_2+1} & & \multicolumn{1}{l}{m_2} & \multicolumn{1}{c}{0_1} & & \multicolumn{1}{l}{k_1} \\
\cline{2-14} 
k_1+1 & \cdots & & & & & & & & & & & & \\ 
 & \cdots & & & & & & & & & & & A^T_1 & \\
m_1 & \cdots & & & & & & & & & 1 & & & \\
\cline{2-14}
0_2 & \cdots & & & & & & & & & & & & \\ 
 & \cdots & & & & & & & & A_2 & & & & \\
k_2 & \cdots & & & & & & & & & & & & \\
\cline{2-14}
k_3+1 & \cdots & & & & & & & & & & & & \\ 
 & \cdots & & & & & A^T_3 & & & & & & & \\
m_3 & \cdots & & & 1 & & & & & & 1 & & & \\
\cline{2-14}
0_4 & \cdots & & & & & & & & & & & & \\ 
 & \cdots & & A_4 & & & & & & & & & & \\
k_4 & \cdots & & & & & & & & & & & & \\
\cline{2-14}
 & \reflectbox{$\ddots$} & & \vdots & & & \vdots & & & \vdots & & & \vdots & \\
\cline{2-14}
\end{array}$$}
\end{proof}

\begin{figure}[h]
\begin{center}
\includegraphics[scale=0.35]{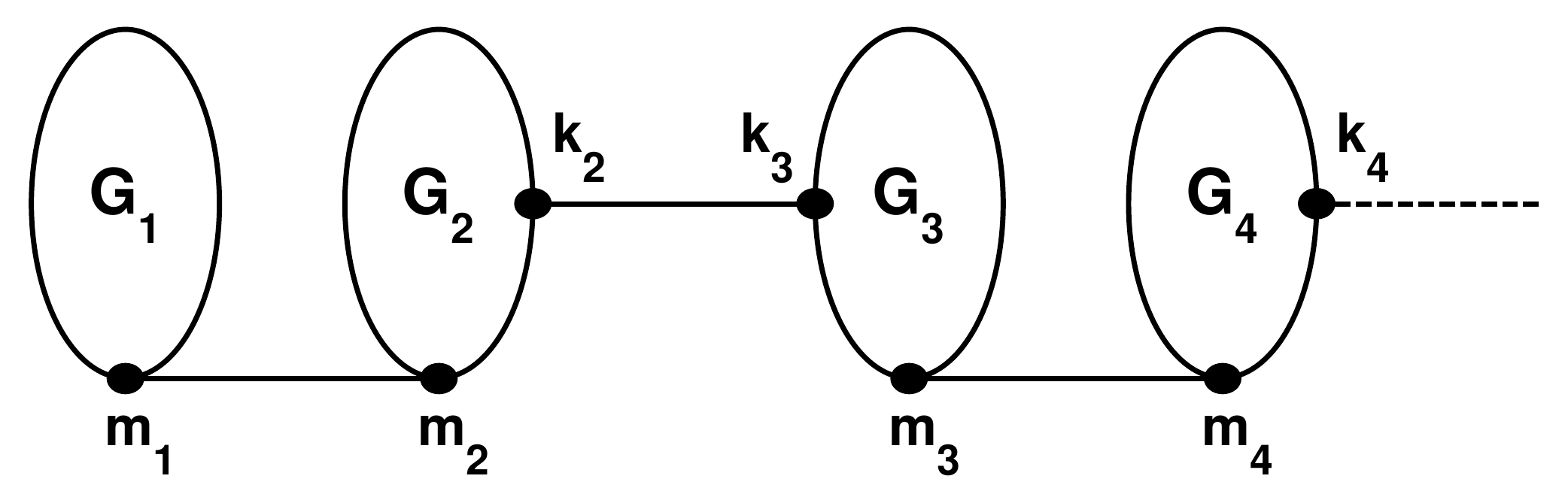}\quad \includegraphics[scale=0.4]{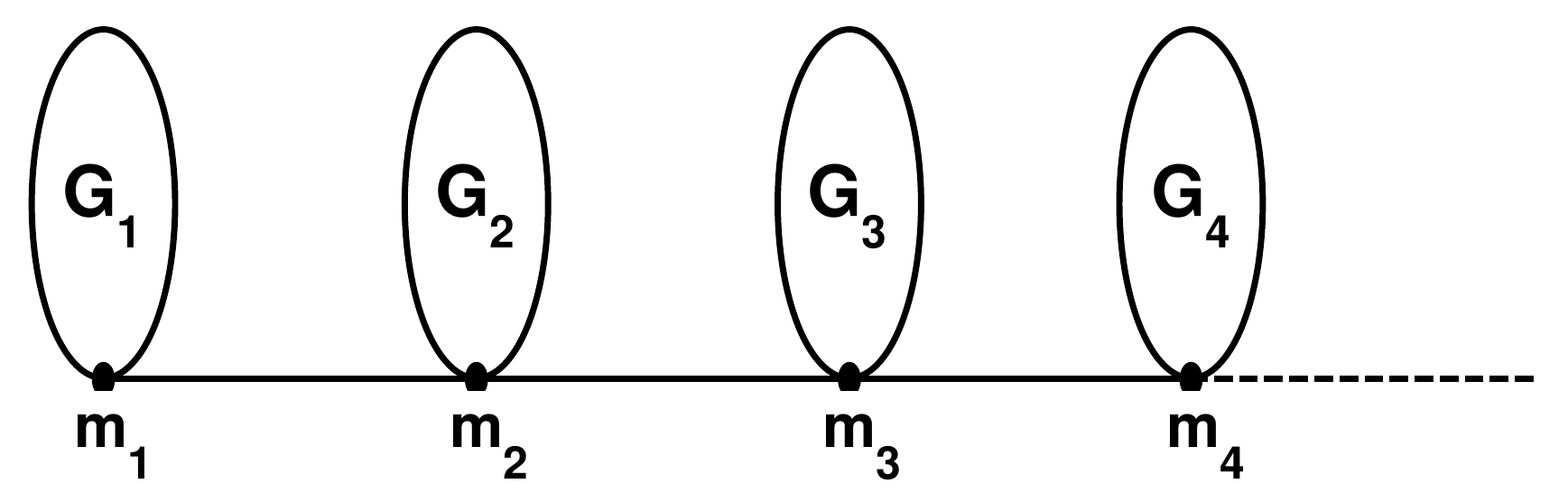}
\caption{The graph $G$ and $H$ in Proposition \ref{p4}}\label{figkm2}
\end{center}
\end{figure}

\begin{cor}\label{cjgr}
Let $\set{G_1,G_2,\ldots,G_r}$ be a collection of completely graceful graphs where $G_i=(V_i,E_i)$ with $|E_i|=m_i$ for all $i=1,2,\ldots,r$ and $V_i\cap V_j=\emptyset$ for all $i\neq j$. For each $i=1,2,\ldots,r-1$, let $G^\prime_i$ be an isomorphic copy of $G_i$. Let us denote a vertex in the graph $G_i$ with the label $\lambda$ by $\lambda_i$ and the corresponding vertex of $G^\prime_i$ by $\lambda^\prime_i$. Let $G=(V,E)$ be a graph obtained by joining the vertex $m_i$ with $m^\prime_i$ and $m^\prime_{i+1}$ for each $i=1,2,\ldots,r-2$ and the vertex $m_{r-1}$ with $m^\prime_{r-1}$ and $m_r$ by edges (cf. Figure \ref{figkm4}). Then $G$ is a completely graceful graph.
\end{cor}

\begin{proof}
Let $\widetilde{G}_i$ be the double of $G_i$ at $m_i$, i.e, $\widetilde{G}_i$ is obtained by joining the vertex $m_i$ of $G_i$ and the vertex $m^\prime_i$ of $G^\prime_i$ with an edge. Then by Corollary \ref{dgrfl}, $\widetilde{G}_i$ has a complete $\alpha$-labeling with a critical number $m_i$. Again it follows from the proof of Proposition \ref{p1} that the maximum labeled vertex in this $\alpha$-labeling is the vertex $m^\prime_i$ of $G_i$ (if we consider the vertex $m_i$ of $G_i$ in the last row and the vertex $m^\prime_i$ of $G^\prime_i$ in the last column of the matrix $A_1$ in the proof of Proposition \ref{p1}). Now we consider the collection $\set{\widetilde{G}_1,\widetilde{G}_2,\ldots,\widetilde{G}_{r-1}}$ and join them to obtain a complete $\alpha$-labeled graph $H$ (say) with a completely graceful biadjacency matrix, say, $A$ as in Proposition \ref{p3} (here $r$ is replaced by $r-1$). Let $A_r$ be a canonical adjacency matrix of $G_r$. Then it is easy to see that the following matrix $M$ is a completely graceful adjacency matrix of $G$. 
{\footnotesize $$M=\quad\begin{array}{l|llr|llr|lll|}
\multicolumn{1}{c}{} & \multicolumn{1}{c}{0^\prime_1} & & \multicolumn{1}{l}{m_{r-1}} &  \multicolumn{1}{c}{0_r} & & \multicolumn{1}{l}{m_r} & \multicolumn{1}{c}{0_{r-1}} & & \multicolumn{1}{l}{m^\prime_1} \\
\cline{2-10} 
0^\prime_1 & & & & & & & & & \\ 
 & & & & & & & & A & \\
m_{r-1} & & & & & & 1 & & & \\
\cline{2-10}
0_r & & & & & & & & & \\ 
 & & & & & A_r & & & & \\
m_r & & & 1 & & & & & & \\
\cline{2-10}
0_{r-1} & & & & & & & & & \\ 
 & & A^T & & & & & & & \\
m^\prime_1 & & & & & & & & & \\
\cline{2-10}
\end{array}$$}

Hence $G$ is completely graceful.
\end{proof}

\begin{figure}[ht]
\begin{center}
\includegraphics[scale=0.4]{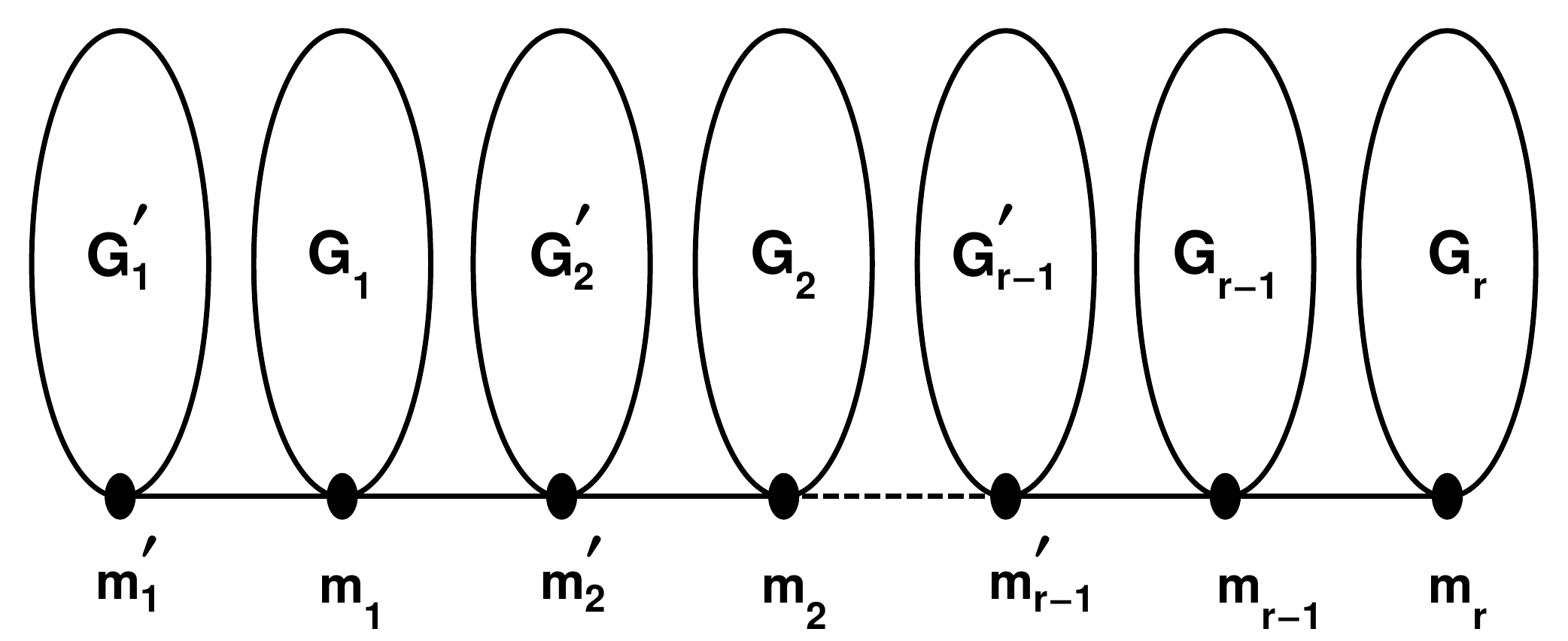}
\caption{The graph $G$ in Corollary \ref{cjgr}}\label{figkm4}
\end{center}
\end{figure}

\begin{prop}\label{newgr11}
Let $\set{G_1,G_2,\ldots,G_r}$ be a collection of completely graceful graphs and $G_i^\prime$ be a copy of $G_i$ for each $i=1,2,\ldots,r-1$ as described in Corollary \ref{cjgr} such that $m_1=m_2=\cdots=m_r$. Let $G=(V,E)$ be a graph obtained by joining a new vertex, say, $v$ with each $m_i$ for $i=1,2,\ldots,r$ and with each $m_i^\prime$ for $i=1,2,\ldots,r-1$ (cf. Figure \ref{figkmv}). Then $G$ is a completely graceful graph with a graceful labeling in which the vertex $v$ has the maximum labeling.
\end{prop}

\begin{proof}
Let $A_i$ be a canonical adjacency matrix of $G_i$ for each $i=1,2,\ldots,r$. Then each $A_i$ is completely graceful and is of same order by the given condition for $i=1,2,\ldots,r$. Consider the following matrix $B$:
{\footnotesize $$B=\quad\begin{array}{c|llr|llr|llr|lll|}
\multicolumn{1}{c}{} & \multicolumn{1}{c}{m_{r-1}^\prime} & & \multicolumn{1}{l}{0_{r-1}} & & \cdots & \multicolumn{1}{r}{} & \multicolumn{1}{c}{m_2^\prime} & & \multicolumn{1}{l}{0_2} & \multicolumn{1}{c}{m_1^\prime} & & \multicolumn{1}{l}{0_1} \\
\cline{2-13} 
m_1 & & & & & & & & & & & & \\ 
 & & & & & \cdots & & & & & & A_1^R & \\
0_1^\prime & & & & & & & & & & & & \\
\cline{2-13}
m_2 & & & & & & & & & & & & \\ 
 & & & & & \cdots & & & A_2^R & & & & \\
0_2^\prime & & & & & & & & & & & & \\
\cline{2-13}
& & & & & & & & & & & & \\ 
\vdots & & \vdots & & & \reflectbox{$\ddots$} & & & \vdots & & & \vdots & \\
 & & & & & & & & & & & & \\
\cline{2-13}
m_{r-1} & & & & & & & & & & & & \\ 
 & & A_{r-1}^R & & & \cdots & & & & & & & \\
0_{r-1}^\prime & & & & & & & & & & & & \\
\cline{2-13}
\end{array}$$}

Then a completely graceful adjacency matrix $M$ of the graph $G$ (cf. Figure \ref{figkmv}) is given below, where the row [column] corresponding to $m_i$ has all its entries $0$ except in the last column [respectively, row] corresponding to the vertex $v$ for all $i=1,2,\ldots,r$, where it is $1$ and the row [column] corresponding to $m_i^\prime$ has all its entries $0$ except in the last column [respectively, row] corresponding to the vertex $v$ for all $i=1,2,\ldots,r-1$, where it is $1$. Moreover these $1$'s are not conflicting with any $1$ in the matrix $B$ as they are lying along those diagonals of $M$ which are either principal diagonals of $A_j^R$ ($j=1,2,\ldots,r-1$) or diagonals formed in between the left bottom corner position of $A_j^R$ and right top corner position of $A_{j+1}^R$ ($j=1,2,\ldots,r-2$) as $m_1=m_2=\cdots=m_r$. For example, the $1$ in the last column of the row corresponding to $m_2$ is along the principal diagonal of $A_1^R$ and the $1$ in the last column of the row corresponding to $m_2$ is along the diagonal of $B$ which begins at the position $m_1m_2^\prime$ and ends at the position $0_2^\prime 0_1$.
 
{\footnotesize $$M=\quad\begin{array}{l|lcrcr|llr|lcccr|c|}
\multicolumn{1}{c}{} & m_1 & m_2 & \cdots & m_{r-1} & \multicolumn{1}{l}{0_{r-1}^\prime} &  m_r & & \multicolumn{1}{l}{0_r} & m_{r-1}^\prime & \cdots & m_2^\prime & m_1^\prime & \multicolumn{1}{l}{0_1} & \multicolumn{1}{l}{v}\\
\cline{2-15} 
m_1 & 0 & & & & & & & & & & & & & 1\\
m_2 & & & & & & & & & & & & & & 1\\ 
\vdots & & & & & & & & & & & B & & &\\
m_{r-1} & & & & & & & & & & & & & & 1\\
0_{r-1}^\prime & & & & & 0 & & & & & & & & &\\
\cline{2-15}
m_r & & & & & & 0 & & & & & & & & 1\\ 
 & & & & & & & A_r^R & & & & & & &\\
0_r & & & & & & & & 0 & & & & & &\\
\cline{2-15}
m_{r-1}^\prime & & & & & & & & & 0 & & & & & 1\\ 
\vdots & & & & & & & & & & & & & &\\ 
m_2^\prime & & & B^T & & & & & & & & & & & 1\\
m_1^\prime & & & & & & & & & & & & & & 1\\ 
0_1 & & & & & & & & & & & & & 0 &\\
\cline{2-15}
v & 1 & 1 & \cdots & 1 & & 1 & & & 1 & \cdots & 1 & 1 & & 0\\
\cline{2-15}
\end{array}$$}
\end{proof}

\vspace{-1.5em}
\begin{figure}[h]
\begin{center}
\includegraphics[scale=0.3]{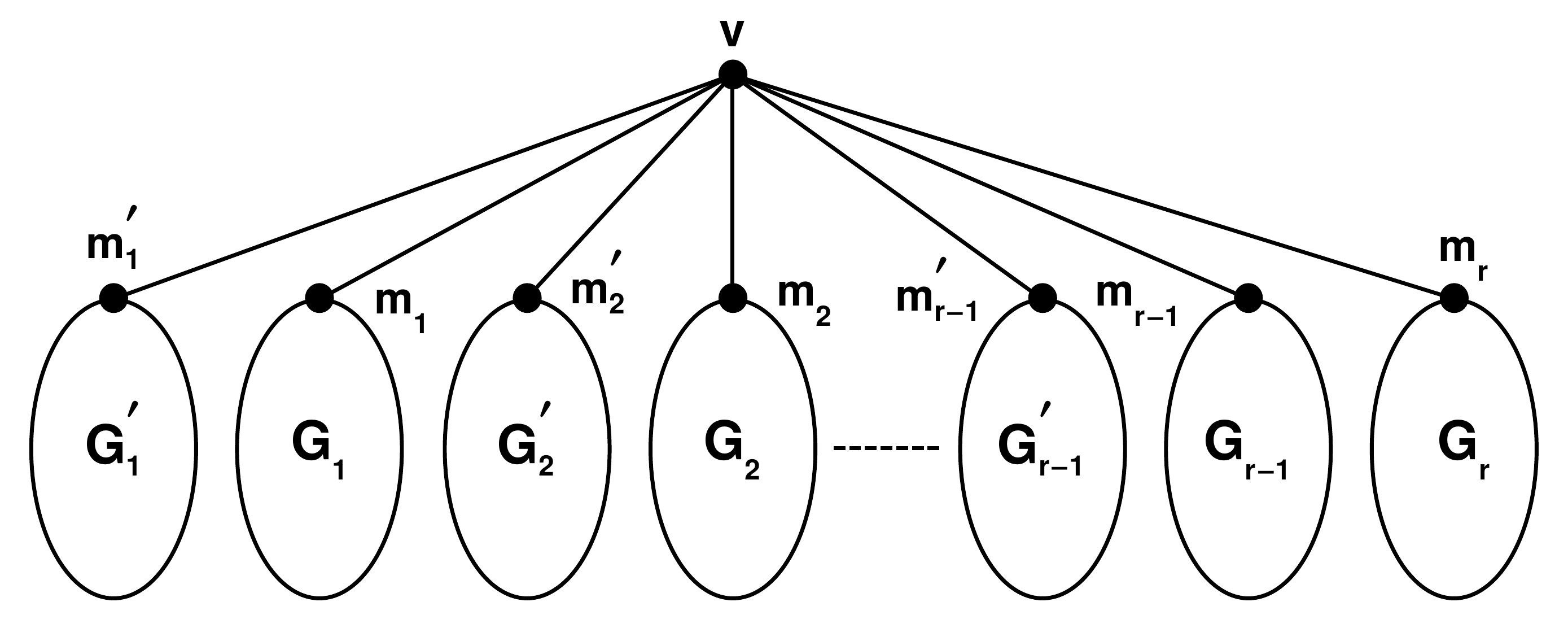}
\caption{The graph $G$ in Proposition \ref{newgr11}}\label{figkmv}
\end{center}
\end{figure}

\vspace{-1.5em}
A tree $T$ is an {\em $F$-tree} if there exists a vertex $v$ of $T$ such that all the branches of $T$ in $v$ are isomorphic, except possibly one branch and each of these branches is a caterpillar. Every $F$-tree is known to be graceful \cite{Rosa}. This result is a special case of Proposition \ref{newgr11}, considering all $G_i$ are isomorphic caterpillars and the fact that any caterpillar can be joined to a graceful tree with the maximum labeled vertex \cite{HH,HKR}.

\vspace{1em}
In \cite{SZ}, Stanton and Zarnke described another method of joining a collection of graceful trees in the following way. Let $S$ and $T$ be two given graceful trees with $V(S)=\set{v_1,v_2,\ldots,v_r}$. For each $i=1,2,\ldots,r$, let $T_i$ be an isomorphic copy of $T$. Then a bigger graceful tree $G$ is obtained by merging the maximum labeled vertex of $T_i$ with $v_i$ for each $i=1,2,\ldots,r$ (i.e., `attaching' a copy of $T$ at each vertex of $S$). The following result is a generalization of this one. In this construction we use adjacency matrices.

\begin{prop}\label{grsz}
Let $H$ be a completely graceful graph with $V(H)=\set{v_0,v_1,v_2,\ldots,v_r}$ such that $v_i$ has the label $i$ for $i=0,1,2,\ldots,r$. Let $\set{G_0,G_1,G_2,\ldots,G_r}$ be a collection of completely graceful graphs, where $G_i=(V_i,E_i)$ with $V_i\cap V(H)=\emptyset$, $|E_i|=m$ for all $i=0,1,2,\ldots,r$ and $V_i\cap V_j=\emptyset$ for all $i\neq j$ such that $G_i\cong G_{r-i}$ for $i=0,1,2,\ldots,\lfloor\frac{r}{2}\rfloor$. Let us denote a vertex in the graph $G_i$ with the label $\lambda$ by $\lambda_i$ for $i=0,1,2,\ldots,r$. Let $G$ be the graph obtained by merging the vertex $m_i$ with $v_i$ for all $i=0,1,2,\ldots,r$. Then $G$ is a completely graceful graph.
\end{prop}

\begin{proof}
Let $A_i$ be a canonical adjacency matrix of $G_i$ for each $i=0,1,2,\ldots,r$. Then each $A_i$ is completely graceful and is of the same order by the given condition for $i=0,1,2,\ldots,r$. Let $B=(b_{ij})_{i,j=0}^r$ be a canonical adjacency matrix of $H$. Consider the following matrix $A$.
{\footnotesize $$A=\quad\begin{array}{c|lcr|lcr|lcr|lcr|lcr|}
\multicolumn{1}{c}{} & 0_r & & \multicolumn{1}{l}{m_r} & & \cdots & \multicolumn{1}{r}{} & 0_2 & & \multicolumn{1}{r}{m_2} & 0_1 & & \multicolumn{1}{l}{m_1} & 0_0 & & \multicolumn{1}{l}{m_0} \\
\cline{2-16} 
0_0 & & & & & & & & & & & & & & & \\ 
 & & & & & \cdots & & & & & & & & & \mathbf{A_0} & \\
m_0 & & & b_{00} & & & & & & b_{0\,(r-2)}& & & b_{0\,(r-1)}& & & b_{0r}\\
\cline{2-16}
0_1 & & & & & & &  & & & & & & & & \\ 
 & & & & & \cdots & & & & & & \mathbf{A_1} & & & & \\
m_1 & & & b_{10} & & & & & & b_{1\,(r-2)}& & & b_{1\,(r-1)}& & & b_{1r}\\
\cline{2-16}
0_2 & & & & & & &  & & & & & & & & \\ 
 & & & & & \cdots & & & \mathbf{A_2} & & & & & & & \\
m_2 & & & b_{20} & & & & & & b_{2\,(r-2)}& & & b_{2\,(r-1)}& & & b_{2r}\\
\cline{2-16}
& & & & & & & & & & & & & & & \\ 
\vdots & & \vdots & & & \reflectbox{$\ddots$} & & & \vdots & & & \vdots & & & \vdots & \\
 & & & & & & & & & & & & & & & \\
\cline{2-16}
0_r & & & & & & & & & & & & & & & \\ 
 & & \mathbf{A_r} & & & \cdots & & & & & & & & & & \\
m_r & & & b_{r0} & & & & & & b_{r\,(r-2)}& & & b_{r\,(r-1)}& & & b_{rr}\\
\cline{2-16}
\end{array}$$}
As in the proof of Proposition \ref{newgr11}, here also the $1$'s in the matrix $B$ are not conflicting with any $1$ in submatrices $A_0,A_1,\ldots,A_r$. Also since $B$ is completely graceful, $A$ is a completely graceful adjacency matrix of $G$.
\end{proof}

We illustrate the above proposition by an example.

\vspace{-0.5em}
\begin{figure}[ht]
\begin{center}
\includegraphics[scale=0.5]{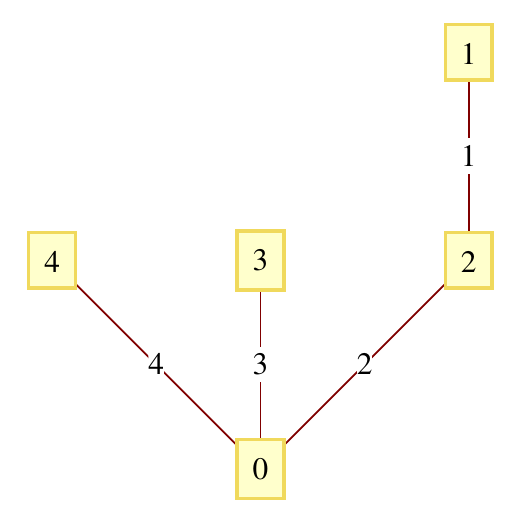}\ \includegraphics[scale=0.5]{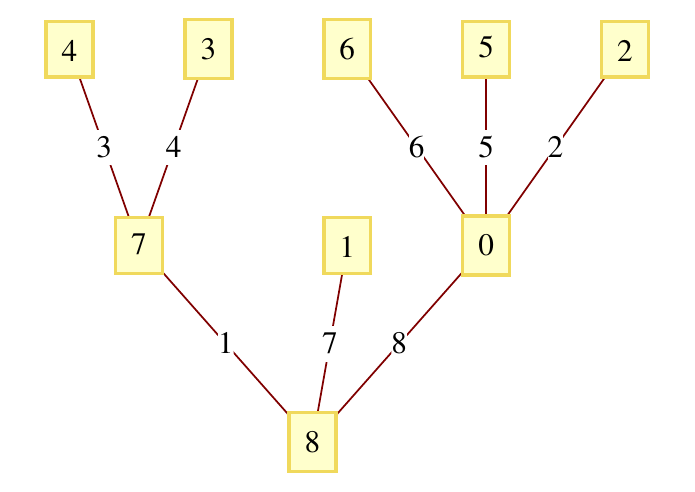}\ \includegraphics[scale=0.5]{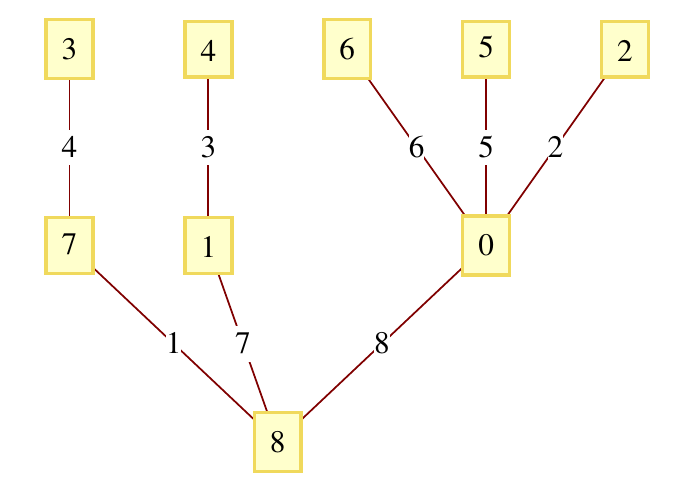}\ \includegraphics[scale=0.5]{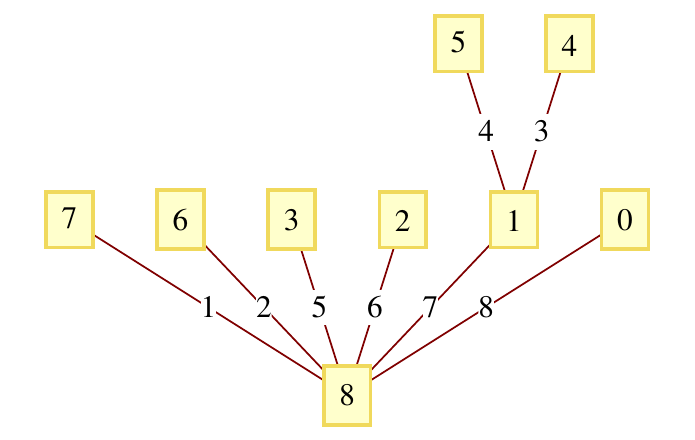}\ 

\vspace{-1em}
\null\hspace{-0.3in} (i) \hspace{1in} (ii) \hspace{1in} (iii) \hspace{1.2in} (iv)
\caption{The tress (i) $H$, (ii) $G_0$, (iii) $G_1$ and (iv) $G_2$ in Example \ref{exgrsz}}\label{figsz}
\end{center}
\end{figure}

\begin{exmp}\label{exgrsz}
Let $H,G_0,G_1,G_2$ be the graceful trees (i), (ii), (iii), (iv) in Figure \ref{figsz} respectively and $G_4\cong G_0, G_3\cong G_1$. The canonical adjacency matrices $A_i$ of $G_i$ ($i=0,1,2$) are given by

{\tiny $A_0=\begin{array}{c|ccccccccc|}
\multicolumn{1}{c}{} & 0 & 1 & 2 & 3 & 4 & 5 & 6 & 7 & \multicolumn{1}{c}{8}\\
\cline{2-10}
0 & 0 & 0 & 1 & 0 & 0 & 1 & 1 & 0 & 1 \\
1 & 0 & 0 & 0 & 0 & 0 & 0 & 0 & 0 & 1 \\
2 & 1 & 0 & 0 & 0 & 0 & 0 & 0 & 0 & 0 \\
3 & 0 & 0 & 0 & 0 & 0 & 0 & 0 & 1 & 0 \\
4 & 0 & 0 & 0 & 0 & 0 & 0 & 0 & 1 & 0 \\
5 & 1 & 0 & 0 & 0 & 0 & 0 & 0 & 0 & 0 \\
6 & 1 & 0 & 0 & 0 & 0 & 0 & 0 & 0 & 0 \\
7 & 0 & 0 & 0 & 1 & 1 & 0 & 0 & 0 & 1 \\
8 & 1 & 1 & 0 & 0 & 0 & 0 & 0 & 1 & 0 \\
\cline{2-10}
\end{array}$,\ $A_1=\begin{array}{c|ccccccccc|}
\multicolumn{1}{c}{} & 0 & 1 & 2 & 3 & 4 & 5 & 6 & 7 & \multicolumn{1}{c}{8}\\
\cline{2-10}
0 & 0 & 0 & 1 & 0 & 0 & 1 & 1 & 0 & 1 \\
1 & 0 & 0 & 0 & 0 & 1 & 0 & 0 & 0 & 1 \\
2 & 1 & 0 & 0 & 0 & 0 & 0 & 0 & 0 & 0 \\
3 & 0 & 0 & 0 & 0 & 0 & 0 & 0 & 1 & 0 \\
4 & 0 & 1 & 0 & 0 & 0 & 0 & 0 & 0 & 0 \\
5 & 1 & 0 & 0 & 0 & 0 & 0 & 0 & 0 & 0 \\
6 & 1 & 0 & 0 & 0 & 0 & 0 & 0 & 0 & 0 \\
7 & 0 & 0 & 0 & 1 & 0 & 0 & 0 & 0 & 1 \\
8 & 1 & 1 & 0 & 0 & 0 & 0 & 0 & 1 & 0 \\
\cline{2-10}
\end{array}$,\ $A_2=\begin{array}{c|ccccccccc|}
\multicolumn{1}{c}{} & 0 & 1 & 2 & 3 & 4 & 5 & 6 & 7 & \multicolumn{1}{c}{8}\\
\cline{2-10}
0 & 0 & 0 & 0 & 0 & 0 & 0 & 0 & 0 & 1 \\
1 & 0 & 0 & 0 & 0 & 1 & 1 & 0 & 0 & 1 \\
2 & 0 & 0 & 0 & 0 & 0 & 0 & 0 & 0 & 1 \\
3 & 0 & 0 & 0 & 0 & 0 & 0 & 0 & 0 & 1 \\
4 & 0 & 1 & 0 & 0 & 0 & 0 & 0 & 0 & 0 \\
5 & 0 & 1 & 0 & 0 & 0 & 0 & 0 & 0 & 0 \\
6 & 0 & 0 & 0 & 0 & 0 & 0 & 0 & 0 & 1 \\
7 & 0 & 0 & 0 & 0 & 0 & 0 & 0 & 0 & 1 \\
8 & 1 & 1 & 1 & 1 & 0 & 0 & 1 & 1 & 0 \\
\cline{2-10}
\end{array}$}

\vspace{1em}\noindent
and $B=$ {\tiny $\begin{array}{c|ccccc|}
\multicolumn{1}{c}{} & 0 & 1 & 2 & 3 & \multicolumn{1}{c}{4}\\
\cline{2-6}
 0 & 0 & 0 & 1 & 1 & 1 \\
 1 & 0 & 0 & 1 & 0 & 0 \\
 2 & 1 & 1 & 0 & 0 & 0 \\
 3 & 1 & 0 & 0 & 0 & 0 \\
 4 & 1 & 0 & 0 & 0 & 0 \\
\cline{2-6}
\end{array}$} be the canonical adjacency matrix of $H$. Let $G$ be the tree obtained 

\vspace{1em}\noindent
by merging the maximum ($8$) labeled vertex of $G_i$ with the vertex labeled $i$ of $H$ for each $i=0,1,2,3,4$ (cf. Figure \ref{exsz2}). Then $G$ is a graceful tree with a completely graceful adjacency matrix $A$, where

{\tiny $$A=\quad\begin{array}{c|c|lcr|lcr|lcr|lcr|lcr|}
\multicolumn{1}{c}{} & \multicolumn{1}{c}{} & 0 & & \multicolumn{1}{l}{8} & 9 & & \multicolumn{1}{r}{17} & 18 & & \multicolumn{1}{r}{26} & 27 & & \multicolumn{1}{l}{35} & 36 & & \multicolumn{1}{l}{44} \\
\cline{3-17}
\multicolumn{1}{c}{} & \multicolumn{1}{c}{} & 0_4 & & \multicolumn{1}{l}{8_4} & 0_3 & & \multicolumn{1}{r}{8_3} & 0_2 & & \multicolumn{1}{r}{8_2} & 0_1 & & \multicolumn{1}{l}{8_1} & 0_0 & & \multicolumn{1}{l}{8_0} \\
\cline{3-17} 
0 & 0_0 & & & & & & & & & & & & & & & \\ 
 & & & & & & & & & & & & & & & \mathbf{A_0} & \\
8 & 8_0 & & & 0 & & & 0 & & & 1 & & & 1 & & & 1\\
\cline{3-17}
9 & 0_1 & & & & & & &  & & & & & & & & \\ 
 & & & & & & & & & & & & \mathbf{A_1} & & & & \\
17 & 8_1 & & & 0 & & & 0 & & & 1 & & & 0 & & & 0 \\
\cline{3-17}
18 & 0_2 & & & & & & &  & & & & & & & & \\ 
& & & & & & & & & \mathbf{A_2} & & & & & & & \\
26 & 8_2 & & & 1 & & & 1 & & & 0 & & & 0 & & & 0 \\
\cline{3-17}
27 & 0_3& & & & & & & & & & & & & & & \\ 
& & & & & & \mathbf{A_1} & & & & & & & & & & \\
35 & 8_3 & & & 1 & & & 0 & & & 0 & & & 0 & & & 0 \\
\cline{3-17}
36 & 0_4 & & & & & & & & & & & & & & & \\ 
 & & & \mathbf{A_0} & & & & & & & & & & & & & \\
44 & 8_4 & & & 1 & & & 0 & & & 0 & & & 0 & & & 0 \\
\cline{3-17}
\end{array}$$}
\end{exmp}

\begin{figure}[ht]
\begin{center}
\includegraphics[scale=0.575]{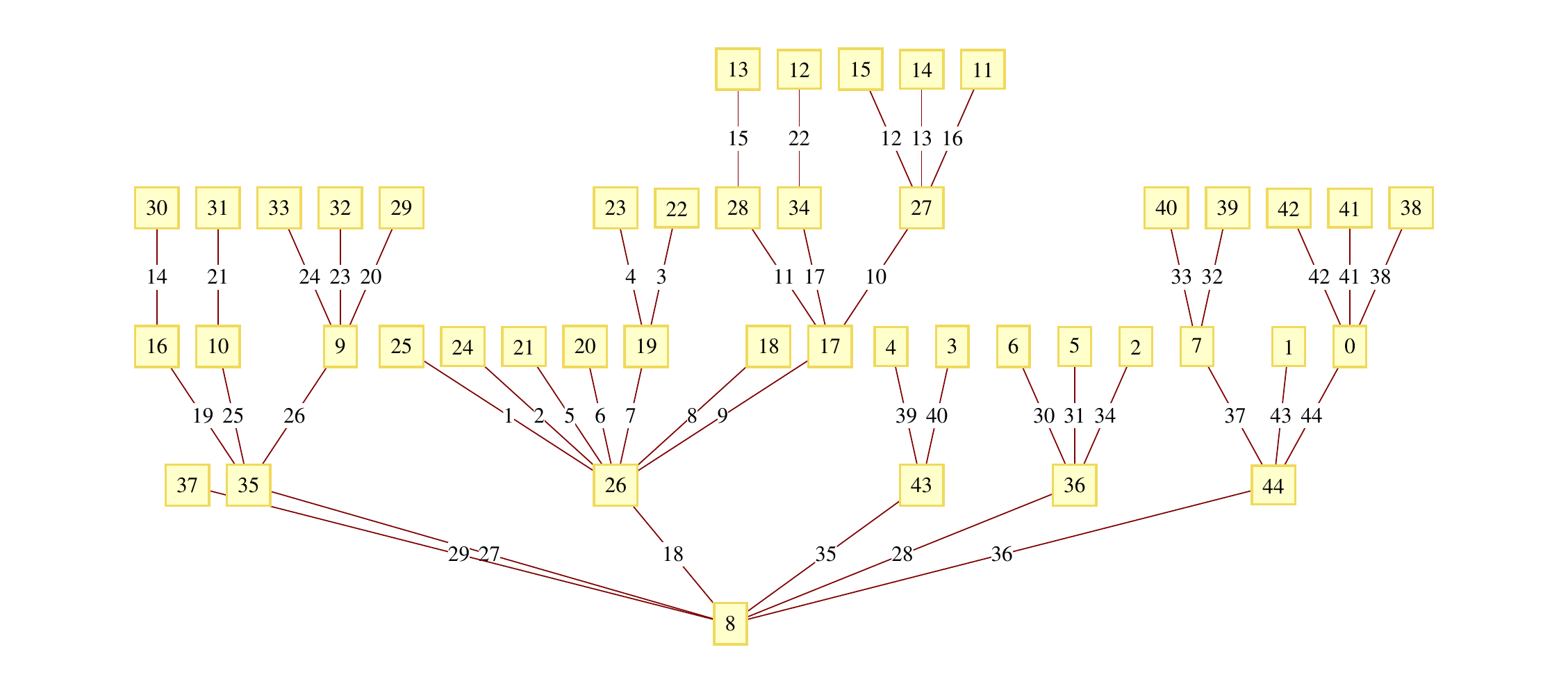} 
\caption{The tree $G$ in Example \ref{exgrsz}}\label{exsz2}
\end{center}
\end{figure}

\begin{rem}\label{rem2}
It is interesting to note that if the completely graceful graph $H$ in Proposition \ref{grsz} has a graceful labeling $f(v_i)=i$, $i=0,1,2,\ldots,r$ such that the neighbors of $v_i$ belong to the set $\Set{v_{r-i-t}}{t\in\set{0,1,2},\ r-i-t\geqslant 0}$ for all $i=0,1,2,\ldots,r$ as in the case of Example \ref{exgrsz}, then the number of vertices of completely graceful graphs $G_0,G_1,G_2,\ldots,G_{\lfloor\frac{r}{2}\rfloor}$ may also be chosen to be different. In fact, Corollary \ref{cjgr} becomes a special case of this when $H$ is a path with $2r-2$ edges and vertices of $H$ are labeled by $(0,2r-2,1,2r-3,2,2r-4,\ldots,r-2,r,r-1)$ which is a graceful labeling.
\end{rem}

\begin{figure}[ht]
\begin{center}
\includegraphics[scale=0.3]{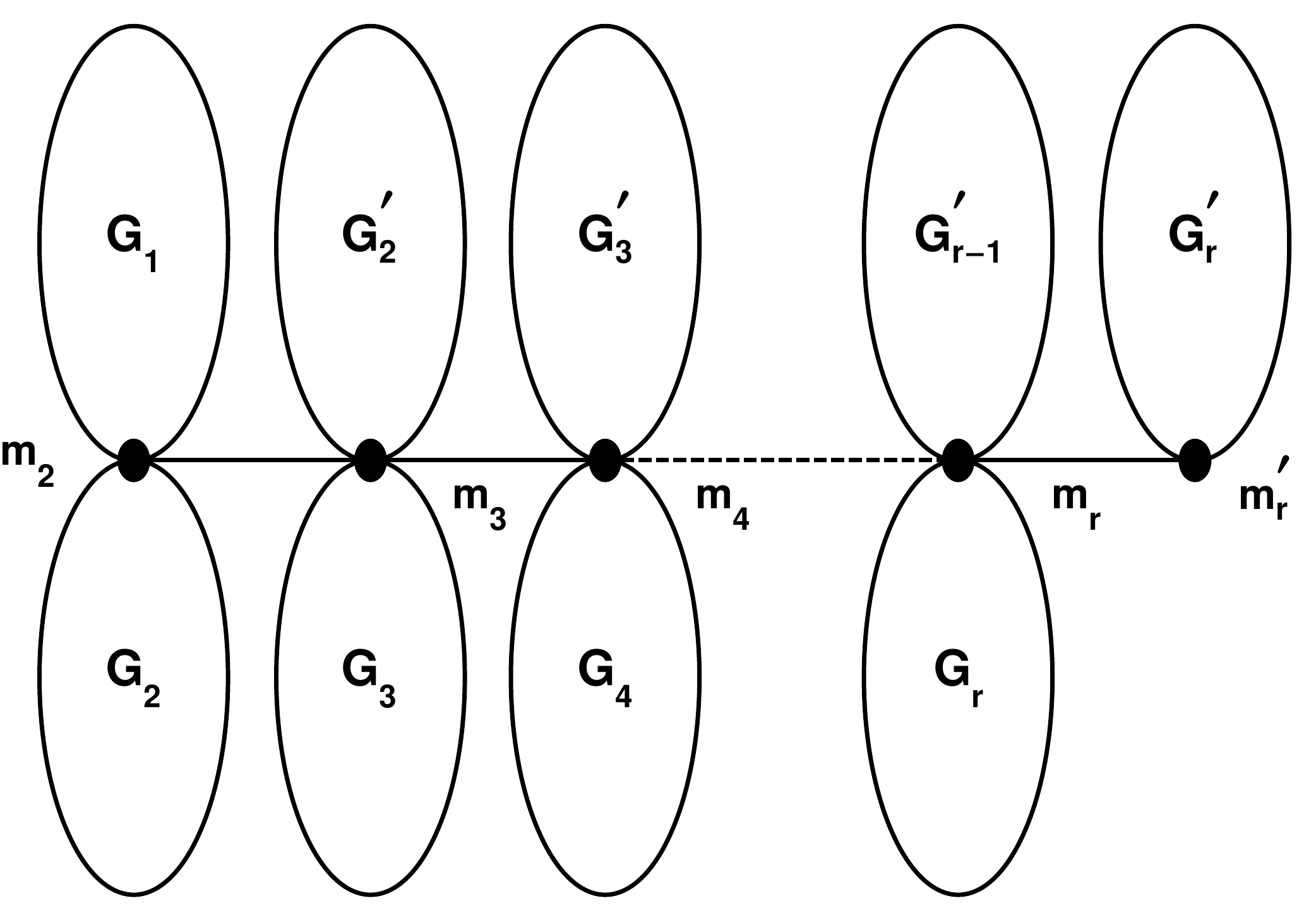}
\caption{The graph $G$ in Proposition \ref{newgr}}\label{figkm5}
\end{center}
\end{figure}

\begin{prop}\label{newgr}
Let $\set{G_1,G_2,\ldots,G_r}$ be a collection of completely graceful graphs as described in Corollary \ref{cjgr}. Let $G=(V,E)$ be a graph obtained by merging the vertex $m_1$ with $m_2$, $m_i^\prime$ with $m_{i+1}$ and joining the vertex $m_r^\prime$ with $m_r$, $m_i$ with $m_{i+1}$ by edges for each $i=2,3,\ldots,r-1$  (cf. Figure \ref{figkm5}). Then $G$ is a completely graceful graph.
\end{prop}

\begin{proof}
Define $\widetilde{G}_i$ as in the proof of Corollary \ref{cjgr} and let $A_i$ be a canonical biadjacency matrix of the complete $\alpha$-labeled graph $\widetilde{G}_i$ for each $i=2,3,\ldots,r$ and $A_1$ be a canonical adjacency matrix of $G_1$. Since $\widetilde{G}_i$ is a double of $G_i$ at $m_i$, there is a $1$ in $A_i$ corresponding to the edge $m_im_i^\prime$ for each $i=2,3,\ldots,r$. Then a completely graceful adjacency matrix $A$ of $G$ is given by
{\footnotesize $$A=\quad\begin{array}{c|c|cc|c|cc|cc|c|cc|cc|c|cc|ccc|c|}
\multicolumn{1}{c}{} & \multicolumn{1}{c}{} &  & \multicolumn{1}{l}{} & \multicolumn{1}{c}{m_3^\prime} &  & \multicolumn{1}{l}{} & & \multicolumn{1}{l}{} & \multicolumn{1}{c}{m_1} &  & \multicolumn{1}{c}{} &  & \multicolumn{1}{l}{} & \multicolumn{1}{l}{m_2^\prime} &  & \multicolumn{1}{l}{} &  & & \multicolumn{1}{l}{m_4^\prime} & \multicolumn{1}{c}{} \\
\multicolumn{1}{c}{} & \multicolumn{1}{c}{} &  & \multicolumn{1}{l}{} & \multicolumn{1}{c}{=} &  & \multicolumn{1}{l}{} & & \multicolumn{1}{l}{} & \multicolumn{1}{c}{=} &  & \multicolumn{1}{c}{} &  & \multicolumn{1}{l}{} & \multicolumn{1}{l}{=} &  & \multicolumn{1}{l}{} &  & & \multicolumn{1}{l}{=} & \multicolumn{1}{c}{} \\
\multicolumn{1}{c}{} & \multicolumn{1}{c}{\cdots} & 0_4^\prime & \multicolumn{1}{l}{} & \multicolumn{1}{c}{m_4} &  & \multicolumn{1}{l}{0_3} & 0_2^\prime & \multicolumn{1}{l}{} & \multicolumn{1}{c}{m_2} &  & \multicolumn{1}{c}{0_1} & 0_2 & \multicolumn{1}{l}{} & \multicolumn{1}{l}{m_3} &  & \multicolumn{1}{l}{0_3^\prime} & 0_4 & & \multicolumn{1}{l}{m_5} & \multicolumn{1}{c}{\cdots} \\
\cline{2-21}
 \vdots & \ddots & & \vdots & \multicolumn{1}{c}{} & \vdots & & & \vdots & \multicolumn{1}{c}{} & \vdots & & \vdots & & \multicolumn{1}{l}{} & \vdots & & & \vdots & & \reflectbox{$\ddots$} \\
\cline{2-21}
0_4^\prime & \cdots & & & \multicolumn{1}{c}{} & & & & & \multicolumn{1}{c}{} & & & & & \multicolumn{1}{l}{} & & & & & & \cdots \\
 & \cdots & & & \multicolumn{1}{c}{} & & & & & \multicolumn{1}{c}{} & & & & & \multicolumn{1}{l}{} & & & & A_4 & & \cdots \\
\cline{2-17}
m_3^\prime = m_4 & \cdots & & & \multicolumn{1}{c}{} & & & & & \multicolumn{1}{c}{} & & & & & \multicolumn{1}{l}{1} & & & & & 1 & \cdots \\
\cline{18-21}
 & \cdots & & & \multicolumn{1}{c}{} & & & & & \multicolumn{1}{c}{} & & & & & \multicolumn{1}{l}{} & A_3^{RT} & & & & & \cdots \\
 0_3 & \cdots & & & \multicolumn{1}{c}{} & & & & & \multicolumn{1}{c}{} & & & & & \multicolumn{1}{l}{} & & & & & & \cdots \\
\cline{2-21}
0_2^\prime & \cdots & & & \multicolumn{1}{c}{} & & & & & \multicolumn{1}{c}{} & & & & \multicolumn{1}{l}{} & & & & & & & \cdots \\
 & \cdots & & & \multicolumn{1}{c}{} & & & & & \multicolumn{1}{c}{} & & & & \multicolumn{1}{l}{A_2} & & & & & & & \cdots \\
\cline{2-12}
m_1=m_2 & \cdots & & & \multicolumn{1}{c}{} & & & & & \multicolumn{1}{c}{0} & & & & \multicolumn{1}{l}{} & 1 & & & & & & \cdots \\
\cline{13-21}
 & \cdots & & & \multicolumn{1}{c}{} & & & & & \multicolumn{1}{c}{} & A_1^R & & & \multicolumn{1}{l}{} & & & & & & & \cdots \\
0_1 & \cdots & & & \multicolumn{1}{c}{} & & & & & \multicolumn{1}{c}{} & & 0 & & \multicolumn{1}{l}{} & & & & & & & \cdots \\
\cline{2-21}
0_2 & \cdots & & & \multicolumn{1}{c}{} & & & & \multicolumn{1}{c}{} & & & & & \multicolumn{1}{l}{} & & & & & & & \cdots \\
 & \cdots & & & \multicolumn{1}{c}{} & & & & \multicolumn{1}{c}{A_2^T} & & & & & \multicolumn{1}{l}{} & & & & & & & \cdots \\
\cline{2-7}
m_2^\prime=m_3 & \cdots & & & \multicolumn{1}{c}{1} & & & & \multicolumn{1}{c}{} & 1 & & & & \multicolumn{1}{l}{} & & & & & & & \cdots \\
\cline{8-21}
 & \cdots & & & \multicolumn{1}{c}{} & A_3^R & & & \multicolumn{1}{c}{} & & & & & \multicolumn{1}{l}{} & & & & & & & \cdots \\
0_3^\prime & \cdots & & & \multicolumn{1}{c}{} & & & & \multicolumn{1}{c}{} & & & & & \multicolumn{1}{l}{} & & & & & & & \cdots \\
\cline{2-21}
0_4 & \cdots & & \multicolumn{1}{c}{} & & & & & \multicolumn{1}{c}{} & & & & & \multicolumn{1}{l}{} & & & & & & & \cdots \\
 & \cdots & & \multicolumn{1}{c}{A_4^T} & & & & & \multicolumn{1}{c}{} & & & & & \multicolumn{1}{l}{} & & & & & & & \cdots \\
m_4^\prime=m_5 & \cdots & & \multicolumn{1}{c}{} & 1 & & & & \multicolumn{1}{c}{} & & & & & \multicolumn{1}{l}{} & & & & & & & \cdots \\
\cline{2-21}
 \vdots & \reflectbox{$\ddots$} & & \multicolumn{1}{c}{\vdots} & & \vdots & & & \multicolumn{1}{c}{\vdots} & & \vdots & & & \multicolumn{1}{l}{\vdots} & & \vdots & & & \vdots & & \ddots \\
\cline{2-21}
\end{array}$$}
\end{proof}

\vspace{-1em}
\section{Certain classes of graceful lobsters}

In general, a lobster has a structure of the graph $L$ in Figure \ref{lob1}(left), where each $F_i$ is a tree of diameter at most $4$. The path $(v_1,v_2,\ldots,v_r)$ is the {\em spine} of the lobster $L$ and each $v_i$ is a {\em spinal vertex} of $L$. Each $F_i$ is a {\em lobe} at the spinal vertex $v_i$ and it has a structure of the graph $F$ in Figure \ref{lob1}(right) (with $v=v_i$), where $s\geqslant 0$, $k_j\geqslant 0$ for all $j=1,2,\ldots,s$. For each $j=1,2,\ldots,s$, the star graph in Figure \ref{lob1}(right) with the central vertex $u_j$ is a {\em branch} at the spinal vertex $v$. For $s=0$, there is no branch at $v$. For any $j=1,2,\ldots,s$, if $k_j=0$, then $u_j$ is a {\em pendant vertex} at $v$. Let $F^\prime_i$ be the tree obtained from $F_i$ by deleting all the pendant vertices at $v_i$. Then $F_i^\prime$ is the {\em reduced lobe} at $v_i$. 

\begin{figure}[h]
\begin{center}
\includegraphics[scale=0.4]{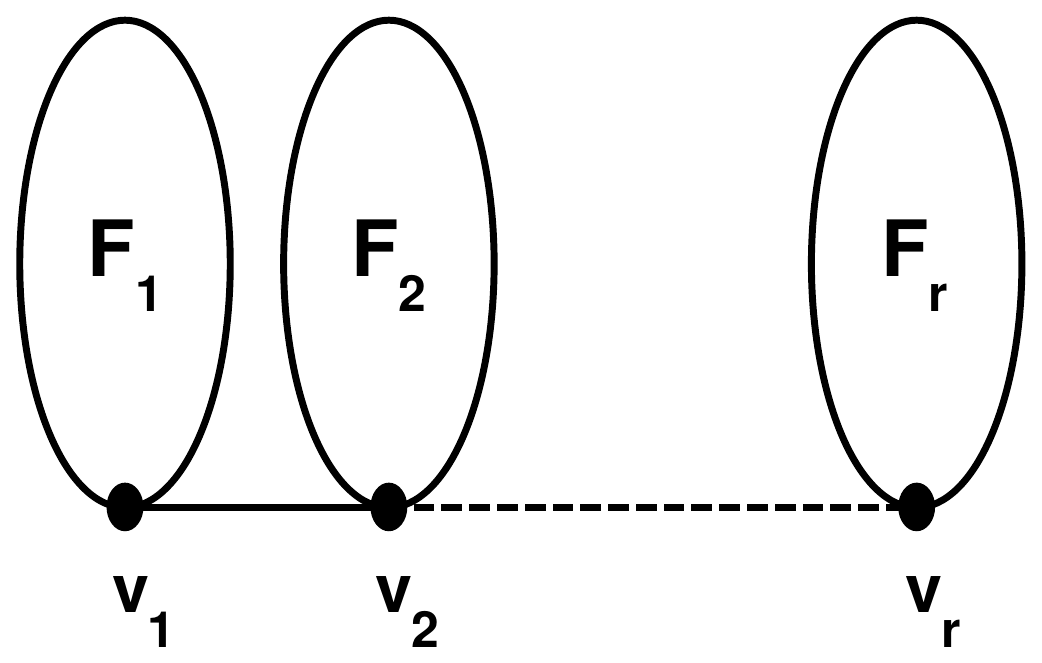}\qquad \includegraphics[scale=0.4]{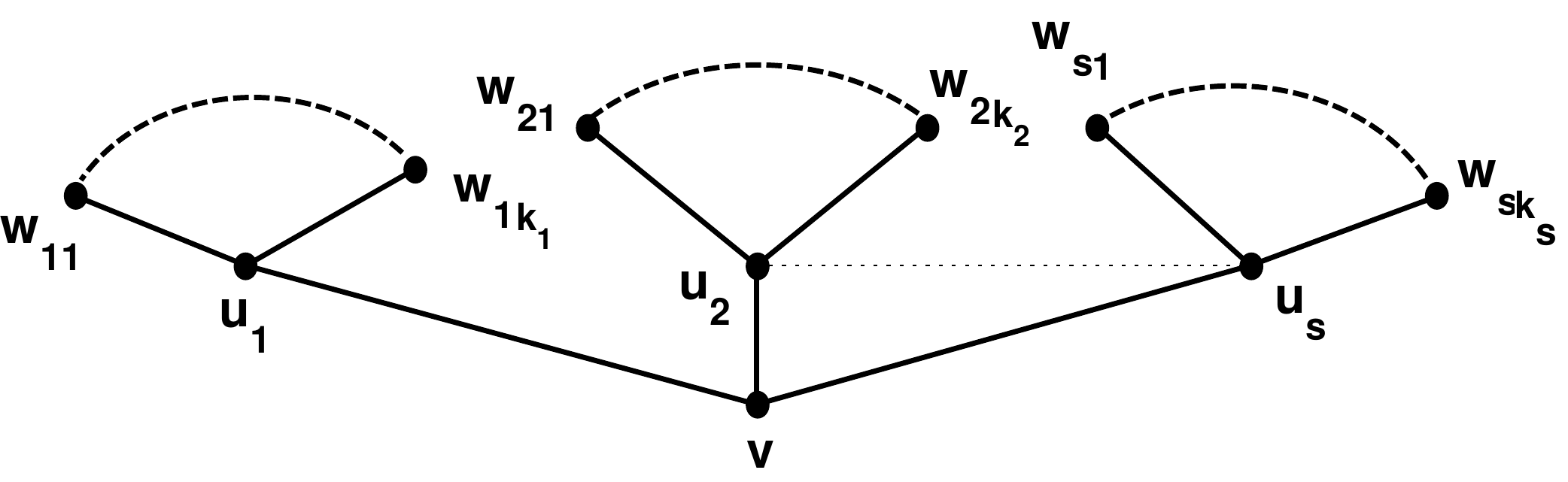}
\caption{A lobster $L$ (left) and a tree $F$ of diameter at most $4$ (right)}\label{lob1}
\end{center}
\end{figure}

\begin{defn}\label{lobd1}
A lobster $L$ in Figure \ref{lob1}(left) is called {\em pairwise isomorphic} if $F_i\cong F_{i+1}$ for all odd $i\in\set{1,2,\ldots,r-1}$. For each $i=1,2,\ldots,r$, let $F^\prime_i$ be the reduced lobe at $v_i$. Then $L$ is called {\em pairwise similar} if $F^\prime_i\cong F^\prime_{i+1}$ for all odd $i\in\set{1,2,\ldots,r-1}$ (cf. Figure \ref{pwslob} for odd $r$. The last lobe will be absent for even $r$). A spinal vertex $v_i$ is called {\em essentially odd} [{\em essentially even}] if the degree of $v_i$ is odd [respectively, even and non-zero] in the tree $F^\prime_i$ for each $i=1,2,\ldots,r$.
\end{defn}

\begin{figure}[ht]
\begin{center}
\includegraphics[scale=0.3]{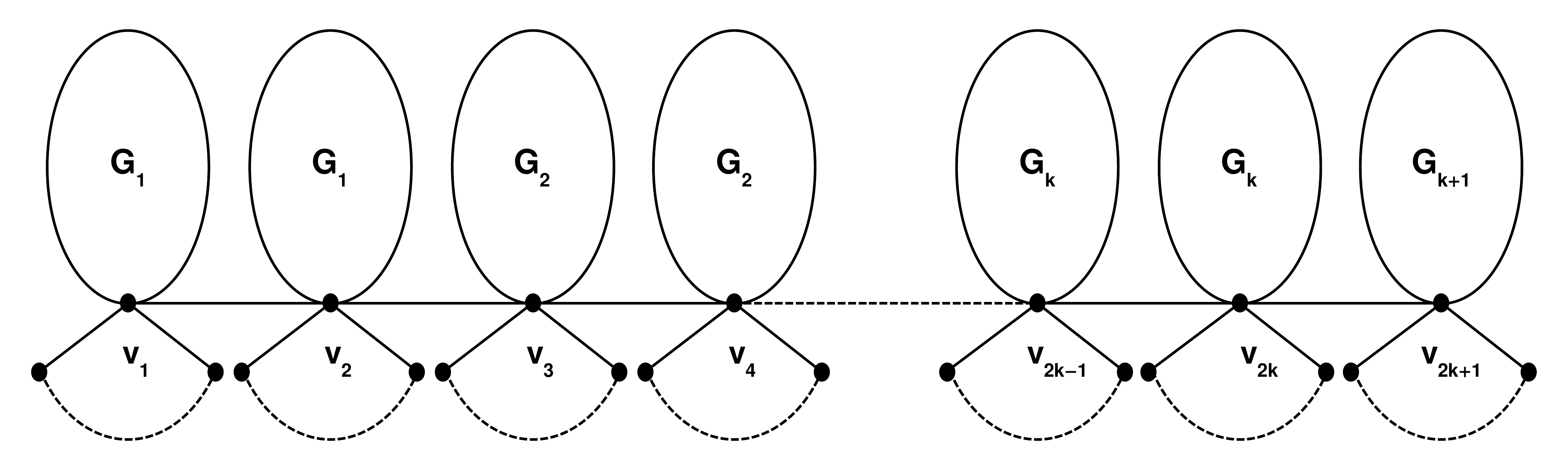}
\caption{A pairwise similar lobster with odd number of spinal vertices}\label{pwslob}
\end{center}
\end{figure}

\begin{thm}\label{lobt1}
A pairwise similar lobster $L$ with essentially odd spinal vertices is graceful.
\end{thm}

\begin{proof}
Let $L$ be a pairwise similar lobster with essetially odd spinal vertices. Let $L^\prime$ be the lobster obtained from $L$ by deleting all the pendant vertices at every spinal vertex of $L$. Now for each odd $i\in\set{1,2,\ldots,r-1}$, consider the reduced lobe $F^\prime_i$ at the spinal vertex $v_i$. Since $F^\prime_i$ is a reduced lobe, the diameter of $F^\prime_i$ can only be $0,2$ or $4$ for otherwise it would invite pendant vertices at the spinal vertex $v_i$ (cf. Figure \ref{lob1}(right)). 
In each of these three cases we show that the following statement is true:

\vspace{0.5em}
{\bf (P1)} $F^\prime_i$ has a graceful labeling such that the label of the spinal vertex is maximum.

\vspace{0.5em}
{\em Case I:} The diameter of $F^\prime_i$ is $0$.

In this case $F^\prime_i$ is the trivial graph with only one vertex, $v_i$. Thus (P1) is true.

\vspace{0.5em}
{\em Case II:} The diameter of $F^\prime_i$ is $2$.

Since $F^\prime_i$ does not contain any pendant vertex at $v_i$, in this case there is only one branch at $v_i$. Then $F^\prime_i$ is a star graph with the central vertex $u_1$. We assign the label $0$ to $u_1$, the label $t$ to $w_{1t}$ for each $t=1,2,\ldots,k_1$ and the label $k_1+1$ to $v_i$ (cf. Figure \ref{lob1}(right) with $v=v_i$). Thus (P1) holds.

\vspace{1em}
{\em Case III:} The diameter of $F^\prime_i$ is $4$.

Then $v_i$ is the central vertex of $F^\prime_i$ (cf. Figure \ref{lob1}(right) with $v=v_i$) and the degree of $v_i$ is odd in $F^\prime_i$ as $v_i$ is essentially odd. Now it follows from the proof of Theorem 1 of \cite{HH}(pg. 137) that every tree of diameter $4$ having the central vertex of an odd degree has a graceful labeling such that the label of the central vertex is maximum. Thus (P1) is true in this case. 

\vspace{0.5em}
Then by Corollary \ref{cjgr}, we have $L^\prime$ is graceful. 

\vspace{0.5em}
Now consider the submatrix $A_i$ of the matrix $A$ in the proof of Corollary \ref{cjgr} (which is also the matrix $A$ in the proof of Proposition \ref{p3}, where $r$ is replaced by $r-1$). In $A_i$ we add $s_i$ number of rows below the row corresponding to $0_i$ with all entries $0$ except those in the last column (i.e., the column corresponding to $m_i$) for which we put $1$. Similarly, we add $t_i$ number of columns after the column corresponding to $k_i+1$ with all entries $0$ except those in the last row (i.e., the row corresponding to $k_i$) for which we put $1$. Suppose the changed matrix be $A^\prime_i$. Then it is clear that $A^\prime_i$ remains to be completely graceful and the graph with adjacency matrix $A^\prime_i$ is obtained from the graph with adjacency matrix $A_i$ by joining $s_i$ one-degree vertices with $m_i$ and $t_i$ one-degree vertices with $k_i$. We do the same for each $i=1,2,\ldots,r-1$ and let the matrix $A$ be changed to $A^\prime$. Finally consider the submatrix $A_r$ of the matrix $M$ in Corollary \ref{cjgr}. We add $t_r$ rows below the row corresponding to $0_r$ and symmetrically $t_r$ columns after the column corresponding to $0_r$ with all entries $0$ excepting $1$ at the last column and the last row respectively. Then the graph with the adjacency matrix $A_r$ is changed to the one with $t_r$ one-degree vertices adjacent to $m_r$. 

\vspace{0.5em} 
Suppose in this way, the matrix $M$ in Corollary \ref{cjgr} is changed to another completely graceful matrix $M^\prime$. Then the graph with adjacency matrix $M^\prime$ will have $s_i$ new one-degree vertices adjacent to $m^\prime_i$ for each $i=1,2,\ldots,r-1$ and $t_i$ new one-degree vertices adjacent to $m_i$ for each $i=1,2,\ldots,r$. Since $s_i$ and $t_i$ are arbitrary, we have $L$ is graceful.
\end{proof}

\begin{cor}\label{lobt2}
A pairwise similar lobster $L$ with at least one pendant vertex at each essentially even spinal vertex is graceful.
\end{cor}

\begin{proof}
We first note that the reduced lobe at each essentially even spinal vertex must be of diameter $4$ as explained in the proof of Theorem \ref{lobt1}. Then by the given condition we can add one pendant vertex at each essentially even spinal vertex so that it is of odd degree in the tree which is its corresponding reduced lobe along with one pendant vertex. Then as in the proof of Theorem \ref{lobt1}, this modified tree is graceful. Then we can join the remaining pendant vertices (if required) as in the proof of Theorem \ref{lobt1} to obtain the given lobster $L$ and $L$ is graceful. 
\end{proof}

Now we wish to present another class of graceful lobsters. Let $G_1=(V_1,E_1)$ and $G_2=(V_2,E_2)$ be two graceful graphs with maximum labeled vertices $m_1$ and $m_2$ respectively and $V_1\cap V_2=\emptyset$. Then the graph $G$ obtained by identifying (or, merging) the vertices $m_1$ and $m_2$ is denoted by $G_1\circ G_2$ (the process is known as {\em gluing} \cite{HKR} of $G_1$ and $G_2$).

\begin{defn}\label{defgr3}
Consider a lobster $L$ in Figure \ref{lob1}(left). For each $i=1,2,\ldots,r$, let $F^\prime_i$ be the reduced lobe at $v_i$.  Let $\set{G_1,G_2,\ldots,G_r}$ be a collection of trees of diameter $4$ where the central vertex is maximum labeled in a graceful labeling of each $G_i$. Then $L$ is called {\em pairwise linked} if $F_i^\prime\cong G_i\circ G_{i+1}$ for each $i=1,2,\ldots,r-1$ and $F_r^\prime\cong G_r$ (cf. Figure \ref{pwl2}). 
\end{defn}

\begin{thm}\label{lobt3}
Every pairwise linked lobster is graceful.
\end{thm}

\begin{proof}
Let $L$ be a pairwise linked lobster. Let $L^\prime$ be the lobster obtained from $L$ by deleting all the pendant vertices at every spinal vertex of $L$. Then by Proposition \ref{newgr}, $L^\prime$ is a graceful lobster. Now using the method similar to the proof of Theorem \ref{lobt1} we can adjoin pendant vertices at the spinal vertices of $L^\prime$ as required so that it remains to be graceful. Thus $L$ is graceful.
\end{proof}

\begin{figure}[h]
\begin{center}
\includegraphics[scale=0.3]{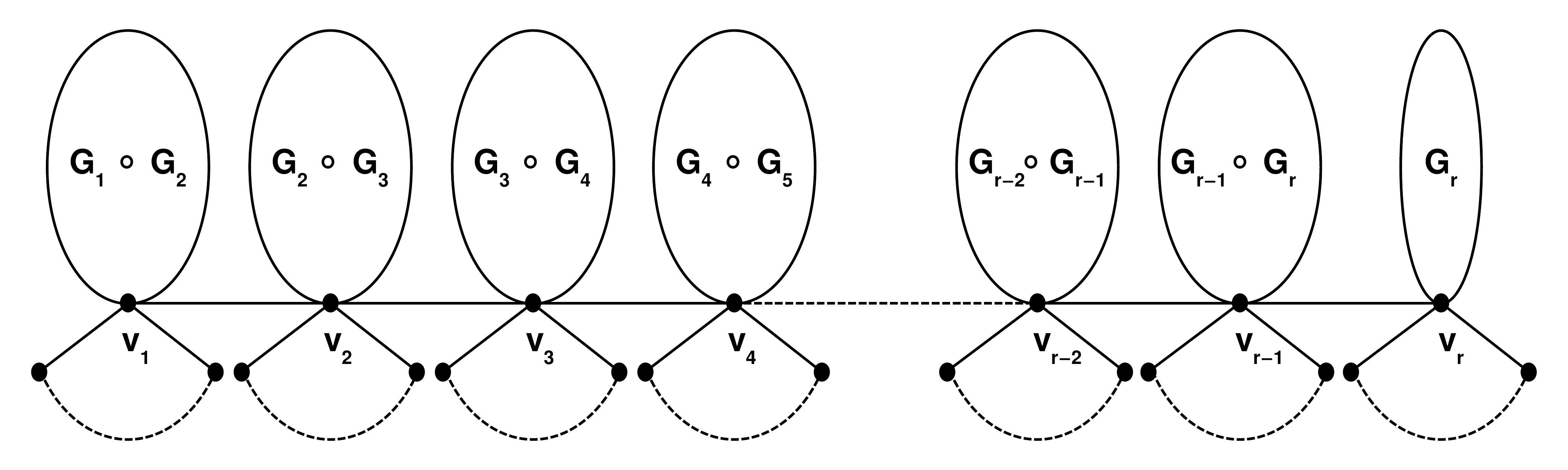}
\caption{A pairwise linked lobster}\label{pwl2}
\end{center}
\end{figure}

\begin{cor}\label{lobcor2}
Let $L$ be a pairwise similar lobster with reduced lobes $F_i^\prime$ at the spinal vertex $v_i$, $i=1,2,\ldots,r$, where $r$ is odd, $v_i$ is essentially even for each $i=1,2,\ldots,r-1$, $F_r^\prime$ consists of a single branch and each $F_i^\prime$ contains a branch isomorphic to $F_r^\prime$ for $i=1,2,\ldots,r-1$ (cf. Figure \ref{pwllob}). Then $L$ is graceful. 
\end{cor}

\begin{figure}[h]
\begin{center}
\includegraphics[scale=0.3]{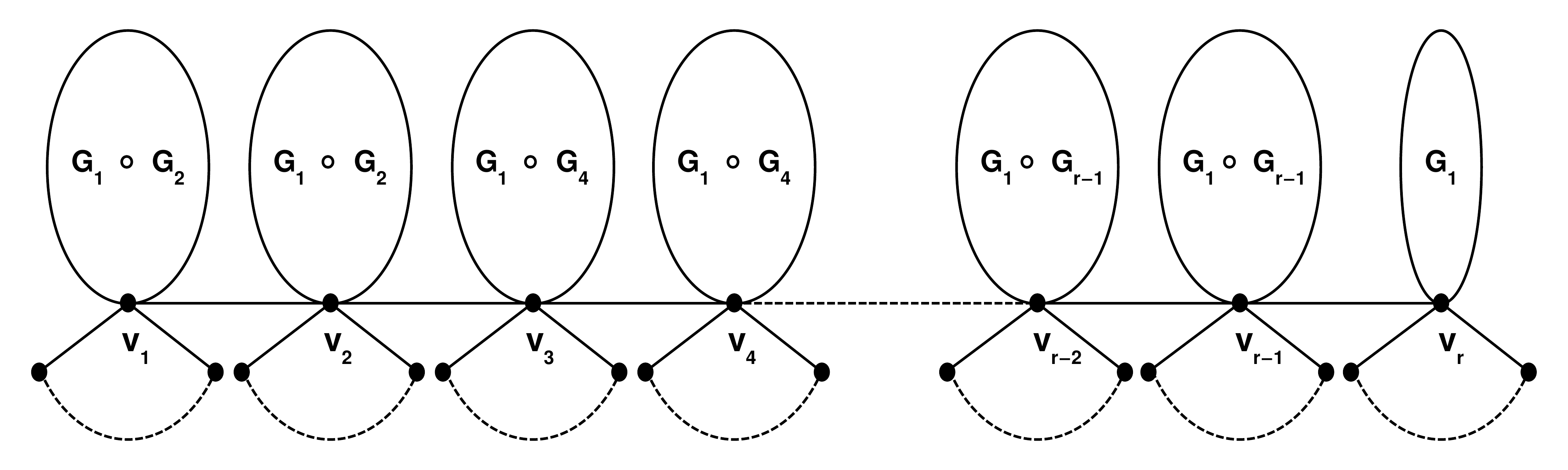}
\caption{The lobster $L$ in Corollary \ref{lobcor2}}\label{pwllob}
\end{center}
\end{figure}

\begin{proof}
For each odd $i\in\set{1,2,\ldots,r-2}$, let $G_{i+1}$ be the tree obtained from $F_i^\prime$ by deleting the branch isomorphic to $F_r^\prime$. Then $\Set{G_{i+1}}{i\text{ is odd, }1\leqslant i\leqslant r-2}$ is a collection of trees where each $G_{i+1}$ is either a tree of diameter $4$ where $v_i$ is the central vertex of $G_{i+1}$ or it is a single branch at $v_i$. In both the cases there is a graceful labeling of $G_{i+1}$ such that $v_i$ has the maximum labeling, as the degree of $v_i$ in $G_{i+1}$ is odd (since $v_i$ is essentially even). Let $G_i$ be isomorphic to $F_r^\prime$ for all odd $i\in\set{1,2,\ldots,r}$. Then $L$ is a pairwise linked lobster with respect to $\set{G_1,G_2,\ldots,G_r}$ (cf. Figure \ref{pwllob}). Thus $L$ is graceful.
\end{proof}

Now let us consider the lobster $G$ of diameter at most $5$ in Figure \ref{pbl}, where the vertex $u_{1i}$ [$u_{2j}$] is adjacent to $x_i$ [respectively, $y_j$] one-degree vertices, $x_i\geqslant 1$, $i=1,2,\ldots,r$, (for $r>0$), $y_j\geqslant 1$, $j=1,2,\ldots,s$ (for $s>0$), the vertex $v_k$  is adjacent to $s_k$ pendant vertices for $k=1,2$ and $r,s,s_1,s_2\geqslant 0$.

\begin{figure}[h]
\begin{center}
\includegraphics[scale=0.25]{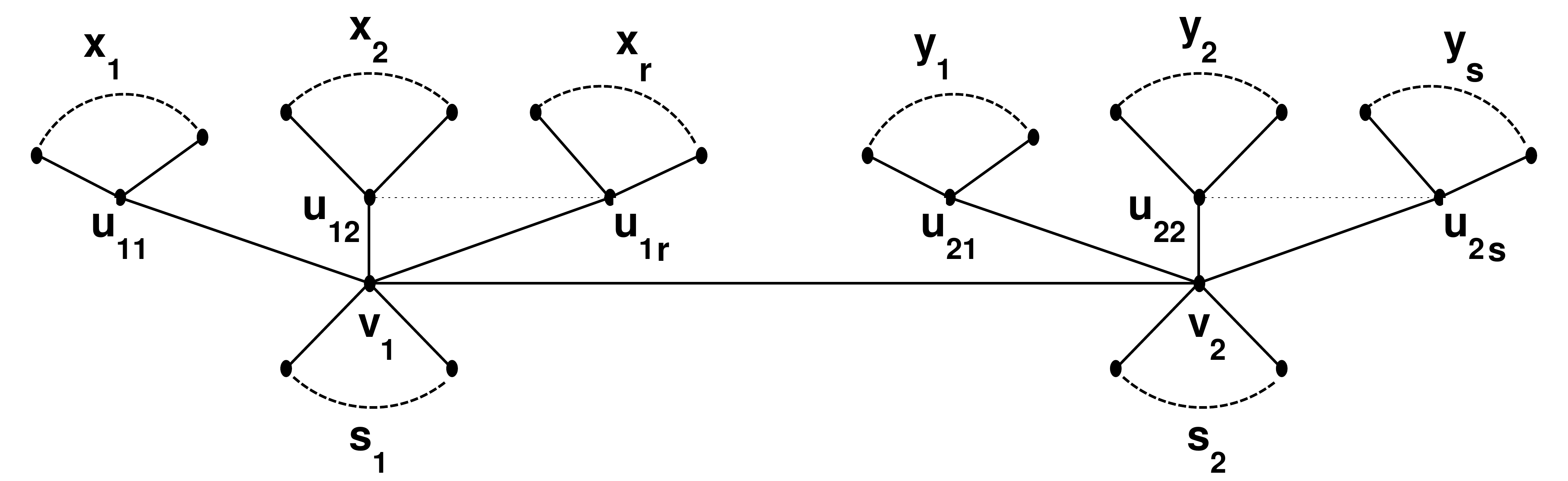}
\caption{A lobster $G$ of diameter at most $5$}\label{pbl}
\end{center}
\end{figure}

\begin{defn}\label{defbal5}
The lobster $G$ in Figure \ref{pbl} is called {\em balanced} if $r=s$ and the following conditions hold (for $r,s>0)$:
\begin{eqnarray}
x_i & = & \left\{%
\begin{array}{ll}
y_{r-\frac{i-1}{2}}, & \text{ when $i$ is odd}\\
x_{\frac{i}{2}}, & \text{ when $i$ is even}
\end{array}\right.\label{eq1}\\
y_i & = & \left\{%
\begin{array}{ll}
x_{r-\frac{i-1}{2}}, & \text{ when $i$ is odd}\\
y_{\frac{i}{2}}, & \text{ when $i$ is even}
\end{array}\right.\label{eq2}
\end{eqnarray}
for all $i=1,2,\ldots,r$. Note that (\ref{eq1}) and (\ref{eq2}) are always satisfied if $x_i=y_i=t$ (say) for all $i=1,2,\ldots,r$. In this case we say that $G$ is {\em trivially balanced}.
\end{defn}

In the following we show that the lobster $G$ in Figure \ref{pbl} has a complete  $\alpha$-labeling if it is balanced.

\begin{lem}\label{pb1}
Let the lobster $G$ in Figure \ref{pbl} be balanced. If $i$ is an odd integer and $j$ is an even integer such that $1\leqslant i,j\leqslant r$. Then {\bf (i)} $\displaystyle{\sum\limits_{t=\frac{i+1}{2}}^i x_t = \sum\limits_{t=r-\frac{i-1}{2}}^r y_t}$, {\bf (ii)}  $\displaystyle{\sum\limits_{t=\frac{i+1}{2}}^i y_t = \sum\limits_{t=r-\frac{i-1}{2}}^r x_t}$,\\
 {\bf (iii)}  $\displaystyle{\sum\limits_{t=\frac{j}{2}+1}^j x_t = \sum\limits_{t=r-\frac{j}{2}+1}^r y_t}$ and {\bf (iv)} $\displaystyle{\sum\limits_{t=\frac{j}{2}+1}^j y_t = \sum\limits_{t=r-\frac{j}{2}+1}^r x_t}$.
\end{lem}

\begin{proof}
We first note that for each $j\in\set{1,3,5,\ldots,i}$ (first $\frac{i+1}{2}$ odd natural numbers), there exists a unique $r\in\Nat$ such that $j\cdot 2^r\in\set{\frac{i+1}{2},\frac{i+3}{2},\ldots,i}$ for if $j\cdot 2^{r-1}<\frac{i+1}{2}$, then $j\cdot 2^r\leqslant i$ and if $j\cdot 2^{r+1}>i$, then $j\cdot 2^r\geqslant\frac{i+1}{2}$, i.e., it cannot happen that $j\cdot 2^{k-1}<\frac{i+1}{2}$ and $j\cdot 2^k>i$ for all $k\in\Nat$. Moreover, if $\frac{i+1}{2}\leqslant j\cdot 2^r\leqslant i$, then $j\cdot 2^{r-1}<\frac{i+1}{2}$ and $j\cdot 2^{r+1}>i$. Now since $x_j=x_{j\cdot 2^k}$ for all $k\in\Nat$ such that $j\cdot 2^k\leqslant r$, we have $\set{x_{\frac{i+1}{2}},x_{\frac{i+3}{2}},\ldots,x_i} =\set{x_1,x_3,\ldots,x_i}= \set{y_r,y_{r-1},\ldots,y_{r-\frac{i-1}{2}}}$. This proves (i). Similarly, we have (ii). Again let $j$ be an even integer such that $1<j\leqslant r$. We have $\displaystyle{\sum\limits_{t=1}^{\frac{j}{2}} x_t}$ $=x_1+x_2+\cdots+x_{\frac{j}{2}}$  $=x_2+x_4+\cdots+x_j$ (as $x_t=x_{2t}$ for all $t\geqslant 1$ such that $2t\leqslant r$). Then 

$\displaystyle{\sum\limits_{t=\frac{j}{2}+1}^j x_t=\sum\limits_{t=1}^j x_t - \sum\limits_{t=1}^{\frac{j}{2}} x_t = x_1+x_3+\cdots+x_{j-1} = y_r + y_{r-1} + y_{r-\frac{j}{2}+1} = \sum\limits_{t=r-\frac{j}{2}+1}^r y_t}$.

The proof of (iv) is similar.
\end{proof}

\begin{thm}\label{pb2}
Let the lobster $G$ in Figure \ref{pbl} be balanced . Then $G$ has a complete $\alpha$-labeling with a critical number $k$, the vertex $v_1$ has the maximum labeling $m$ and $v_2$ has the labeling $k$, where  $\displaystyle{k=s_1+r+\sum\limits_{i=1}^r y_i}$ and $\displaystyle{m=s_1+s_2+2r+1+\sum\limits_{i=1}^r (x_i+y_i)}$.
\end{thm}

\begin{proof}
Consider the lobster $G$ in Figure \ref{pbl}. Let the pendant vertices at $v_i$ be $\gamma_{i1},\gamma_{i2},\ldots,\gamma_{is_i}$ for $i=1,2$, the one-degree vertices adjacent to $u_{1j}$ be $\alpha_{j1},\alpha_{j2},\ldots,\alpha_{jx_j}$ and the one-degree vertices adjacent to $u_{2j}$ be $\beta_{j1},\beta_{j2},\ldots,\beta_{jy_j}$ for all $j=1,2,\ldots,r$. Let us arrange the vertices of $G$ in rows and columns of a matrix $A$ in the following order and label them from $0$ to $k$ according to the increasing order of rows and $k+1$ to $m$ according to the increasing order of columns.

{\footnotesize $$\begin{array}{|c|c|c|c|c|c|c|c|c|c|c|c|}
\hline
\text{Vertices} & \gamma_{11} & \cdots & \gamma_{1s_1} & u_{1r} & \beta_{11} & \cdots & \beta_{1y_1} & u_{1(r-1)} & \beta_{21} & \cdots & \beta_{2y_2}\\
\hline
\text{Labels} & 0 & \cdots & s_1-1 & s_1 & s_1+1 & \cdots & s_1+y_1 & s_1+y_1+1 & s_1+y_1+2 & \cdots & s_1+y_1+y_2+1\\
\hline
\end{array}$$}
\null\hfill $\cdots$ \hfill\null
{\footnotesize $$\begin{array}{|c|c|c|c|c|c|c|}
\hline
\text{Vertices} & \beta_{j1} & \beta_{jy_j} & u_{1(r-j)} & \cdots & \beta_{ry_r} & v_2\\
\hline
\text{Labels} & s_1+j+\sum\limits_{t=1}^{j-1} y_t & s_1+j-1+\sum\limits_{t=1}^{j} y_t & s_1+j+\sum\limits_{t=1}^{j} y_t & \cdots & s_1+r-1+\sum\limits_{t=1}^{r} y_t & k=s_1+r+\sum\limits_{t=1}^{r} y_t\\
\hline
\end{array}$$}
{\footnotesize $$\begin{array}{|c|c|c|c|c|c|c|c|c|c|c|}
\hline
\text{Vertices} & \gamma_{21} & \cdots & \gamma_{2s_2} & u_{2r} & \alpha_{11} & \cdots & \alpha_{1x_1} & u_{2(r-1)} & \cdots & \alpha_{j1}\\
\hline
\text{Labels} & k+1 & \cdots & k+s_2 & k+s_2+1 & k+s_2+2 & \cdots & k+s_2+x_1+1 & k+s_2+x_1+2 & \cdots & k+s_2+j+1+\sum\limits_{t=1}^{j-1} x_t\\
\hline
\end{array}$$}
\null\hfill $\cdots$ \hfill\null
{\footnotesize $$\begin{array}{|c|c|c|c|c|c|}
\hline
\text{Vertices} & \alpha_{jx_j} & u_{2(r-j)} & \cdots & \alpha_{rx_r} & v_1\\
\hline
\text{Labels} & k+s_2+j+\sum\limits_{t=1}^{j} x_t & k+s_2+j+1+\sum\limits_{t=1}^{j} x_t & \cdots & k+s_2+r+\sum\limits_{t=1}^{r} x_t & m=k+s_2+r+1+\sum\limits_{t=1}^{r} x_t\\
\hline
\end{array}$$}

Then by using Lemma \ref{pb1} one can verify (though it is  rigorous) that the above labeling is a complete $\alpha$-labeling of $G$ with the following induced edge labeling where all egde labels are distinct.

{\footnotesize $$\begin{array}{|c|c|c|c|c|c|c|c|c|c|c|}
\hline
\text{Edges} & v_2\gamma_{21} & \cdots & v_2\gamma_{2s_2} & v_2u_{2r} & u_{2r}\beta_{ry_r} & \cdots & u_{2r}\beta_{r1} & v_2u_{2(r-1)} & u_{11}\alpha_{11} & \cdots \\
\hline
\text{Labels} & 1 & \cdots & s_2 & s_2+1 & s_2+2 & \cdots & s_2+y_r+1 & s_2+x_1+2 & s_2+y_r+3 & \cdots \\
 & & & & & & & & =s_2+y_r+2 & & \\
 & & & & & & & (\text{for }r>1) & (\text{ as }x_1=y_r) & & \\
\hline
\end{array}$$}
{\footnotesize $$\begin{array}{|c|c|c|c|c|c|c|}
\hline
\text{Edges} & u_{11}\alpha_{1x_1} & v_2u_{2(r-2)} & u_{2(r-1)}\beta_{(r-1)y_{r-1}} & \cdots & u_{2(r-1)}\beta_{(r-1)1} & \cdots \\
\hline
\text{Labels} & s_2+y_r+x_1+2 & s_2+x_1+x_2+3 & s_2+y_r+x_1+4 & \cdots & s_2+y_r+y_{r-1}+x_1+3 &  \cdots \\
 & & =s_2+y_r+x_1+3 & & & & \\
 & (\text{for }r>2) & (\text{ as }x_1=x_2=y_r) & & & & \\
\hline
\end{array}$$}
{\footnotesize $$\begin{array}{|c|c|c|c|c|}
\hline
\text{Edges} & v_2u_{2(r-j)} & u_{2(r-\frac{j}{2})}\beta_{(r-\frac{j}{2})y_{r-\frac{j}{2}}} & u_{1(\frac{j+1}{2})}\alpha_{(\frac{j+1}{2})1} & \cdots \\
\hline
\text{Labels} & s_2+j+1+\sum\limits_{t=1}^{j} x_t & s_2+j+2+\sum\limits_{t=1}^{\frac{j}{2}} x_t +\sum\limits_{t=r-\frac{j}{2}+1}^r y_t & s_2+j+2+\sum\limits_{t=1}^{\frac{j-1}{2}} x_t +\sum\limits_{t=r-\frac{j-1}{2}}^r y_t & \cdots \\
 & & =s_2+j+2+\sum\limits_{t=1}^j x_t & =s_2+j+2+\sum\limits_{t=1}^j x_t & \\
 & (\text{for }r>j) & (\text{by Lemma \ref{pb1}, for }j \text{ is even}) & (\text{by Lemma \ref{pb1}, for }j \text{ is odd}) & \\
\hline
\end{array}$$}
{\footnotesize $$\begin{array}{|c|c|c|c|c|}
\hline
\text{Edges} & v_2v_1 & u_{2(\frac{r}{2})}\beta_{(\frac{r}{2})y_{\frac{r}{2}}} & u_{1(\frac{r+1}{2})}\alpha_{(\frac{r+1}{2})1} & \cdots \\
\hline
\text{Labels} & s_2+r+1+\sum\limits_{t=1}^{r} x_t & s_2+r+2+\sum\limits_{t=1}^{\frac{r}{2}} x_t +\sum\limits_{t=\frac{r}{2}+1}^r y_t & s_2+r+2+\sum\limits_{t=1}^{\frac{r-1}{2}} x_t +\sum\limits_{t=\frac{r+1}{2}}^r y_t & \cdots \\
 & & =s_2+r+2+\sum\limits_{t=1}^r x_t & =s_2+r+2+\sum\limits_{t=1}^r x_t & \\
 & & (\text{by Lemma \ref{pb1}, for }r \text{ is even}) & (\text{by Lemma \ref{pb1}, for }r \text{ is odd}) & \\
\hline
\end{array}$$}
{\footnotesize $$\begin{array}{|c|c|c|c|c|}
\hline
\text{Edges} & u_{11}v_1 & u_{1(\frac{r}{2}+1)}\alpha_{(\frac{r}{2}+1)1} & u_{2(\frac{r-1}{2})}\beta_{(\frac{r-1}{2})y_{\frac{r-1}{2}}} & \cdots \\
\hline
\text{Labels} & s_2+r+2+y_r+\sum\limits_{t=1}^{r} x_t & s_2+r+3+\sum\limits_{t=1}^{\frac{r}{2}} x_t +\sum\limits_{t=\frac{r}{2}}^r y_t & s_2+r+3+\sum\limits_{t=1}^{\frac{r+1}{2}} x_t +\sum\limits_{t=\frac{r+1}{2}}^r y_t & \cdots \\
 & & =s_2+r+3+y_{\frac{r}{2}}+\sum\limits_{t=1}^r x_t & =s_2+r+3+x_{\frac{r+1}{2}}+\sum\limits_{t=1}^r x_t & \\
 & & =s_2+r+3+y_r+\sum\limits_{t=1}^r x_t & =s_2+r+3+y_r+\sum\limits_{t=1}^r x_t & \\
 & & (\text{by Lemma \ref{pb1} and (\ref{eq2}), for }r \text{ is even}) & (\text{by Lemma \ref{pb1} and (\ref{eq2}), for }r \text{ is odd}) & \\
\hline
\end{array}$$}
{\footnotesize $$\begin{array}{|c|c|c|c|c|}
\hline
\text{Edges} & u_{1i}v_1 & u_{1(\frac{r+i+1}{2})}\alpha_{(\frac{r+i+1}{2})1} & u_{2(\frac{r-i}{2})}\beta_{(\frac{r-i}{2})y_{\frac{r-i}{2}}} & \cdots \\
\hline
\text{Labels} & s_2+r+i+1+\sum\limits_{t=1}^{r} x_t +\sum\limits_{t=r-i+1}^r y_t & s_2+r+i+2+\sum\limits_{t=1}^{\frac{r+i-1}{2}} x_t +\sum\limits_{t=\frac{r-i+1}{2}}^r y_t & s_2+r+i+2+\sum\limits_{t=1}^{\frac{r+i}{2}} x_t +\sum\limits_{t=\frac{r-i}{2}+1}^r y_t & \cdots \\
 & & =s_2+r+i+2+\sum\limits_{t=1}^{r} x_t +\sum\limits_{t=r-i+1}^r y_t & =s_2+r+i+2+\sum\limits_{t=1}^{r} x_t +\sum\limits_{t=r-i+1}^r y_t & \\
 & & (\text{by Lemma \ref{pb1}(ii), for }r+i \text{ is even}) & (\text{by Lemma \ref{pb1}(iv), for }r+i \text{ is odd}) & \\
\hline
\end{array}$$}
{\footnotesize $$\begin{array}{|c|c|c|c|c|}
\hline
\text{Edges} & u_{1r}v_1 & \gamma_{1s_1}v_1 & \cdots & \gamma_{11}v_1 \\
\hline
\text{Labels} & s_2+2r+1+\sum\limits_{t=1}^r (x_t+y_t) & s_2+2r+2+\sum\limits_{t=1}^r (x_t+y_t) & \cdots & s_1+s_2+2r+1+\sum\limits_{t=1}^r (x_t+y_t) =m\\
\hline
\end{array}$$}

Moreover, the last row of the biadjacency matrix $A$ corresponds to the vertex $v_2$ which has a label $k$ and the last column of $A$ corresponds to the vertex $v_1$ which has the maximum label $m$. Thus $G$ has a complete $\alpha$-labeling with a critical number $k$, the vertex $v_1$ has the maximum labeling $m$ and $v_2$ has the labeling $k$
\end{proof}

\begin{figure}[h]
\begin{center}
\includegraphics[scale=0.5]{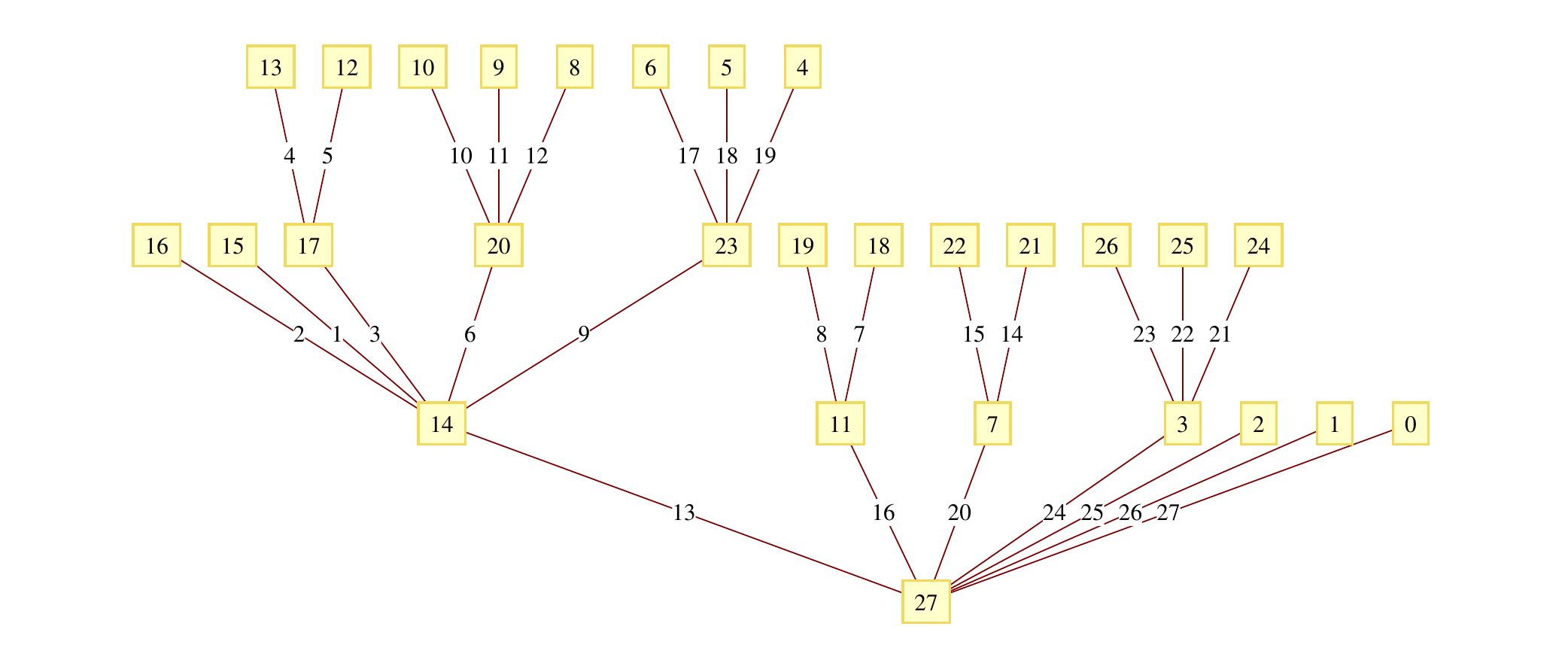}
\caption{The lobster $G$ in Example \ref{pbex} with a complete $\alpha$-labeling}\label{pbl2}
\end{center}
\end{figure}

\begin{exmp}\label{pbex}
Let us consider the lobster $G$ in Figure \ref{pbl} with $r=3$, $x_1=x_2=y_3=2$, $y_1=y_2=x_3=3$, $s_1=3$ and $s_2=2$ (cf. Figure \ref{pbl2}). Then the following is the completely graceful biadjacency matrix $A$ of $G$ as constructed in the proof of Theorem \ref{pb2}.
{\footnotesize $$A\ =\ \begin{array}{cc|ccccccccccccc|}
 & \multicolumn{1}{c}{} & \multicolumn{1}{c}{15} & 16 & 17 & 18 & 19 & 20 & 21 & 22 & 23 & 24 & 25 & 26 & \multicolumn{1}{l}{27}\\
 & \multicolumn{1}{c}{} & \multicolumn{1}{c}{\gamma_{21}} & \gamma_{22} & u_{23} & \alpha_{11} & \alpha_{12} & u_{22} & \alpha_{21} & \alpha_{22} & u_{21} & \alpha_{31} & \alpha_{32} & \alpha_{33} & \multicolumn{1}{l}{v_1}\\
\cline{3-15}
0 & \gamma_{11} & 0 & 0 & 0 & 0 & 0 & 0 & 0 & 0 & 0 & 0 & 0 & 0 & 1 \\
1 & \gamma_{12} & 0 & 0 & 0 & 0 & 0 & 0 & 0 & 0 & 0 & 0 & 0 & 0 & 1 \\
2 & \gamma_{13} & 0 & 0 & 0 & 0 & 0 & 0 & 0 & 0 & 0 & 0 & 0 & 0 & 1 \\
3 & u_{13} & 0 & 0 & 0 & 0 & 0 & 0 & 0 & 0 & 0 & 1 & 1 & 1 & 1 \\
4 & \beta_{11} & 0 & 0 & 0 & 0 & 0 & 0 & 0 & 0 & 1 & 0 & 0 & 0 & 0 \\
5 & \beta_{12} & 0 & 0 & 0 & 0 & 0 & 0 & 0 & 0 & 1 & 0 & 0 & 0 & 0 \\
6 & \beta_{13} & 0 & 0 & 0 & 0 & 0 & 0 & 0 & 0 & 1 & 0 & 0 & 0 & 0 \\
7 & u_{12} & 0 & 0 & 0 & 0 & 0 & 0 & 1 & 1 & 0 & 0 & 0 & 0 & 1 \\
8 & \beta_{21} & 0 & 0 & 0 & 0 & 0 & 1 & 0 & 0 & 0 & 0 & 0 & 0 & 0 \\
9 & \beta_{22} & 0 & 0 & 0 & 0 & 0 & 1 & 0 & 0 & 0 & 0 & 0 & 0 & 0 \\
10 & \beta_{23} & 0 & 0 & 0 & 0 & 0 & 1 & 0 & 0 & 0 & 0 & 0 & 0 & 0 \\
11 & u_{11}& 0 & 0 & 0 & 1 & 1 & 0 & 0 & 0 & 0 & 0 & 0 & 0 & 1 \\
12 & \beta_{31} & 0 & 0 & 1 & 0 & 0 & 0 & 0 & 0 & 0 & 0 & 0 & 0 & 0 \\
13 & \beta_{32} & 0 & 0 & 1 & 0 & 0 & 0 & 0 & 0 & 0 & 0 & 0 & 0 & 0 \\
14 & v_2 & 1 & 1 & 1 & 0 & 0 & 1 & 0 & 0 & 1 & 0 & 0 & 0 & 1 \\
\cline{3-15}
\end{array}$$}
\end{exmp}

\begin{defn}\label{lobpb2}
A lobster $L$ in Figure \ref{lob1}(left) is called {\em pairwise balanced} if $r$ is even and the subtree of $L$ consisting of vertices $v_i$ and $v_{i+1}$ along with lobes $F_i$ and $F_{i+1}$ is balanced for each odd $i\in\set{1,2,\ldots,r-1}$.
\end{defn}

\begin{cor}\label{pbt}
Every pairwise balanced lobster has a complete $\alpha$-labeling.
\end{cor}

\begin{proof}
The result follows from Proposition \ref{p3} and Theorem \ref{pb2}.
\end{proof}

A very special case of pairwise balanced lobsters is particularly interesting.

\begin{defn}\label{defpb3}
Let $L$ be a pairwise balanced lobster with lobes $F_1,F_2,\ldots,F_r$ at the spinal vertices $v_1,v_2,\ldots,v_r$ respectively such that for each odd $i\in\set{1,2,\ldots,r-1}$, the subtree of $L$ consisting of vertices $v_i$ and $v_{i+1}$ along with lobes $F_i$ and $F_{i+1}$ is trivially balanced (i.e., $F_i$ and $F_{i+1}$ have same numbers of isomorphic (non-pendant) branches). Then $L$ is called {\em pairwise trivially balanced} lobster. 
\end{defn}

\begin{cor}\label{pbt2}
Any pairwise trivially balanced lobster has a complete $\alpha$-labeling.
\end{cor}

\begin{figure}[ht]
\begin{center}
\includegraphics[scale=0.56]{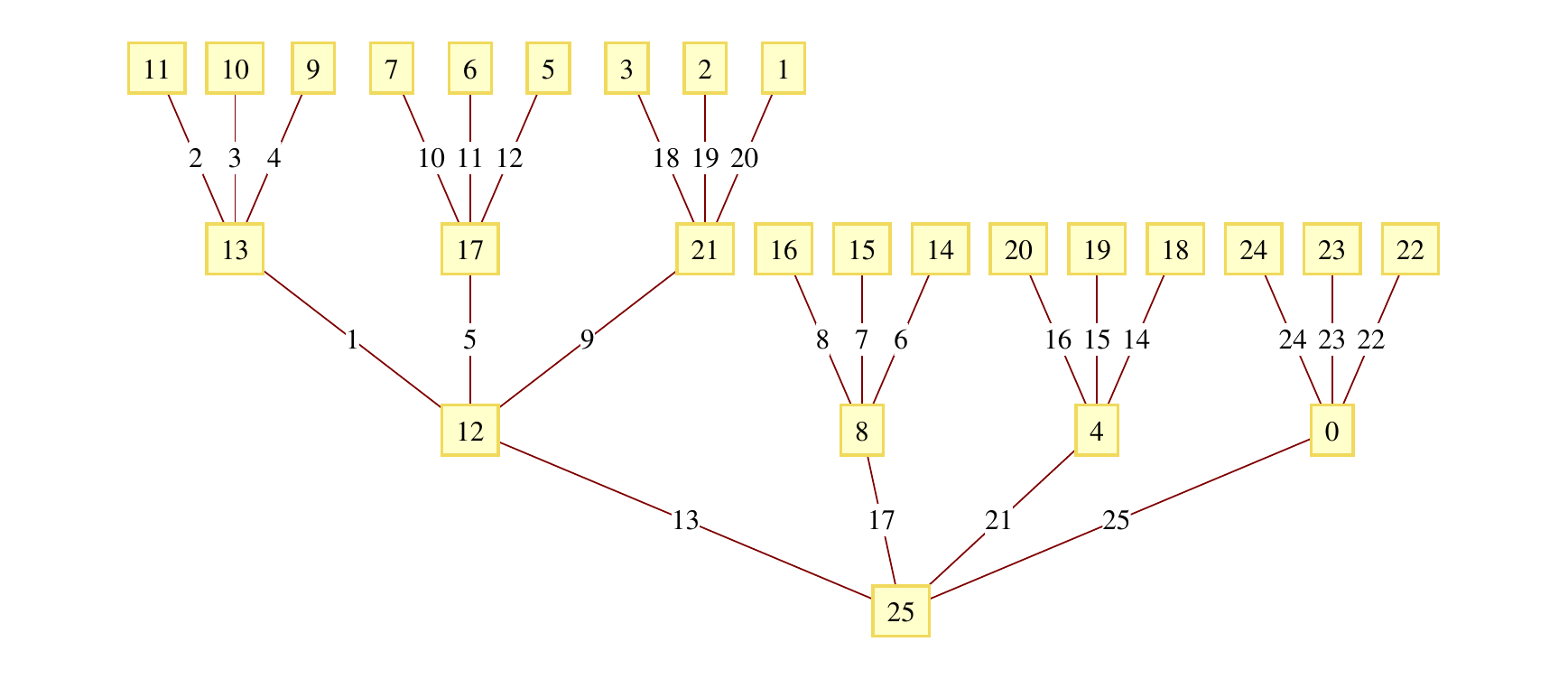}
\caption{The trivially balanced lobster $T$ in Example \ref{newex3}}\label{newexfig3}
\end{center}
\end{figure}

\begin{exmp}\label{newex3}
Consider the trivially balanced lobster $T$ in Figure \ref{newexfig3}. The completely graceful biadjacency matrix $A$ of $T$ is given by 

{\footnotesize $$A\ =\ 
\begin{array}{c|ccccccccccccc|}
\multicolumn{1}{c}{} & 13 & 14 & 15 & 16 & 17 & 18 & 19 & 20 & 21 & 22 & 23 & 24 & \multicolumn{1}{c}{25} \\
\cline{2-14}
0 & 0 & 0 & 0 & 0 & 0 & 0 & 0 & 0 & 0 & 1 & 1 & 1 & 1 \\
1 & 0 & 0 & 0 & 0 & 0 & 0 & 0 & 0 & 1 & 0 & 0 & 0 & 0 \\
2 & 0 & 0 & 0 & 0 & 0 & 0 & 0 & 0 & 1 & 0 & 0 & 0 & 0 \\
3 & 0 & 0 & 0 & 0 & 0 & 0 & 0 & 0 & 1 & 0 & 0 & 0 & 0 \\
4 & 0 & 0 & 0 & 0 & 0 & 1 & 1 & 1 & 0 & 0 & 0 & 0 & 1 \\
5 & 0 & 0 & 0 & 0 & 1 & 0 & 0 & 0 & 0 & 0 & 0 & 0 & 0 \\
6 & 0 & 0 & 0 & 0 & 1 & 0 & 0 & 0 & 0 & 0 & 0 & 0 & 0 \\
7 & 0 & 0 & 0 & 0 & 1 & 0 & 0 & 0 & 0 & 0 & 0 & 0 & 0 \\
8 & 0 & 1 & 1 & 1 & 0 & 0 & 0 & 0 & 0 & 0 & 0 & 0 & 1 \\
9 & 1 & 0 & 0 & 0 & 0 & 0 & 0 & 0 & 0 & 0 & 0 & 0 & 0 \\
10 & 1 & 0 & 0 & 0 & 0 & 0 & 0 & 0 & 0 & 0 & 0 & 0 & 0 \\
11 & 1 & 0 & 0 & 0 & 0 & 0 & 0 & 0 & 0 & 0 & 0 & 0 & 0 \\
12 & 1 & 0 & 0 & 0 & 1 & 0 & 0 & 0 & 1 & 0 & 0 & 0 & 1 \\
\cline{2-14}
\end{array}$$}

The importance of this special case is that by shifting the $1$'s in its completely graceful biadjacency matrix one can obtain many other graceful lobsters which are not pairwise similar, linked or balanced. For example let us shift the following $1$'s in $A$ (keeping the same set of box-values) to obtain a new matrix, say, $A^\prime$. Then $A^\prime$ remains to be a completely graceful and it is the biadjacency matrix of another lobster, say, $T^\prime$ (cf. Figure \ref{newexfig4}) which has also a complete $\alpha$-labeling. But $T^\prime$ is not pairwise similar, linked or balanced.
{\footnotesize $$\begin{array}{|c|c||c|c|}
\hline
\text{Shift of $1$'s} & \text{Change of box-values} & \text{Shift of $1$'s} & \text{Change of box-values}\\
\hline
\set{1,21}\longrightarrow \set{1,17} & 20\longrightarrow 16 & \set{4,20}\longrightarrow \set{0,20} & 16\longrightarrow 20\\
\hline
\set{2,21}\longrightarrow \set{2,17} & 19\longrightarrow 15 & 
\set{4,19}\longrightarrow \set{0,19} & 15\longrightarrow 19\\
\hline
\set{5,17}\longrightarrow \set{5,13} & 12\longrightarrow 8 & 
\set{8,16}\longrightarrow \set{4,16} & 8\longrightarrow 12\\
\hline
\end{array}$$}

{\footnotesize $$A^\prime\ =\ 
\begin{array}{c|ccccccccccccc|}
\multicolumn{1}{c}{} & 13 & 14 & 15 & 16 & 17 & 18 & 19 & 20 & 21 & 22 & 23 & 24 & \multicolumn{1}{c}{25} \\
\cline{2-14}
0 & 0 & 0 & 0 & 0 & 0 & 0 & \mathbf{1} & \mathbf{1} & 0 & 1 & 1 & 1 & 1 \\
1 & 0 & 0 & 0 & 0 & \mathbf{1} & 0 & 0 & \mathbf{0} & 0 & 0 & 0 & 0 & 0 \\
2 & 0 & 0 & 0 & 0 & \mathbf{1} & 0 & 0 & \mathbf{0} & 0 & 0 & 0 & 0 & 0 \\
3 & 0 & 0 & 0 & 0 & 0 & 0 & 0 & 0 & 1 & 0 & 0 & 0 & 0 \\
4 & 0 & 0 & 0 & \mathbf{1} & 0 & 1 & \mathbf{0} & \mathbf{0} & 0 & 0 & 0 & 0 & 1 \\
5 & \mathbf{1} & 0 & 0 & 0 & \mathbf{0} & 0 & 0 & 0 & 0 & 0 & 0 & 0 & 0 \\
6 & 0 & 0 & 0 & 0 & 1 & 0 & 0 & 0 & 0 & 0 & 0 & 0 & 0 \\
7 & 0 & 0 & 0 & 0 & 1 & 0 & 0 & 0 & 0 & 0 & 0 & 0 & 0 \\
8 & 0 & 1 & 1 & \mathbf{0} & 0 & 0 & 0 & 0 & 0 & 0 & 0 & 0 & 1 \\
9 & 1 & 0 & 0 & 0 & 0 & 0 & 0 & 0 & 0 & 0 & 0 & 0 & 0 \\
10 & 1 & 0 & 0 & 0 & 0 & 0 & 0 & 0 & 0 & 0 & 0 & 0 & 0 \\
11 & 1 & 0 & 0 & 0 & 0 & 0 & 0 & 0 & 0 & 0 & 0 & 0 & 0 \\
12 & 1 & 0 & 0 & 0 & 1 & 0 & 0 & 0 & 1 & 0 & 0 & 0 & 1 \\
\cline{2-14}
\end{array}.$$}
\end{exmp}

\vspace{-1em}
\begin{figure}[h]
\begin{center}
\includegraphics[scale=0.6]{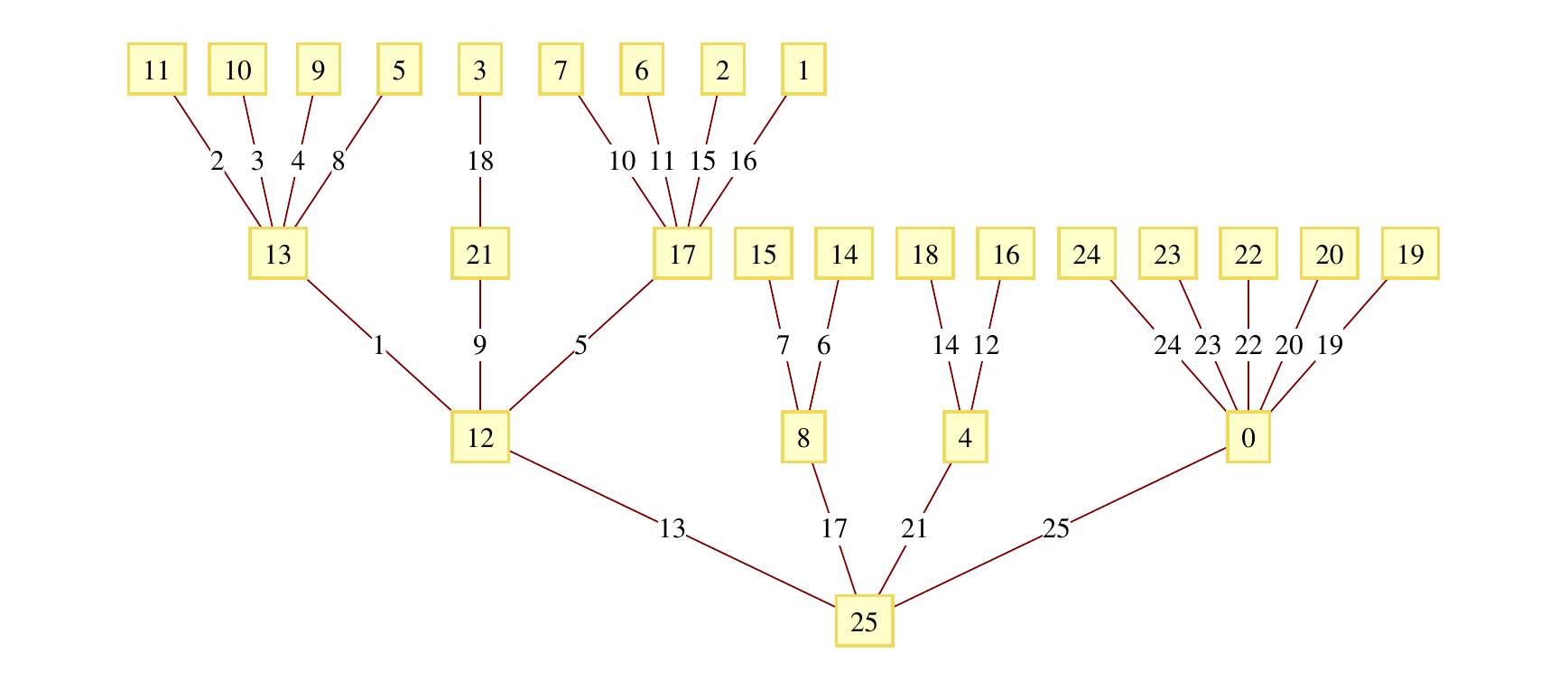}
\caption{The lobster $T^\prime$ in Example \ref{newex3}}\label{newexfig4}
\end{center}
\end{figure}

\section{Conclusion}

The conjecture that all trees are graceful is approaching its $50$ years and the subsequent partial conjecture that all lobsters are graceful is also open for nearly $35$ years. From the above study of finding graceful lobsters the following questions are now becoming important. 

\begin{enumerate}
\item In \cite{HH} it was shown that all trees (lobsters) of diameter $5$ are graceful. Does there exist an $\alpha$-labeling of a lobster $G$ in Figure \ref{pbl} such that the central vertices $v_1$ and $v_2$ are labeled by the critical number and the maximum labeling?

\item Again in \cite{HH} it was shown that any tree of diameter $4$ having a central vertex of an odd degree has a graceful labeling such that the label of the central vertex is maximum. Is there any such graceful labeling for a tree of diameter $4$ having a central vertex of an even degree?
\end{enumerate}

If answers to both the questions are affirmative, then the lobster conjecture can be solved by joining these trees as in Proposition \ref{p3}. If the answer to the second question is yes, then all pairwise similar lobsters become graceful without any restriction. However, in search of more new classes of graceful lobsters, an immediate attention to Example \ref{newex3} will be helpful as it indicates another possible class of graceful lobsters whose pairwise lobes satisfy the property that $r=s$, $\displaystyle{\sum\limits_{t=1}^r x_t=\sum\limits_{t=1}^r y_t}=q$ (say) and $r$ divides $q$.

\bibliographystyle{amsplain}

\end{document}